\documentclass{amsart}
\usepackage{amsmath,amssymb,amsfonts,amsthm}%
\usepackage{bm}
\usepackage{hyperref}
\usepackage{enumitem}%
\usepackage[nameinlink,noabbrev]{cleveref}

\newtheorem{thm}{Theorem}[section]
\newtheorem{ineq}{Inequality}
\newtheorem{lem}[thm]{Lemma}
\newtheorem{conj}[thm]{Conjecture}
\newtheorem{prop}[thm]{Proposition}
\newtheorem{cor}[thm]{Corollary}

\theoremstyle{definition}
\newtheorem{rem}[thm]{Remark}
\newtheorem{example}[thm]{Example}

\numberwithin{equation}{section}

\let\min\relax \DeclareMathOperator*\min{\vphantom{p}min}
\let\max\relax \DeclareMathOperator*\max{\vphantom{p}max}
\begin{document}
\begin{sloppypar}

\title[On the monotonicity of left and right Riemann sums]{On the monotonicity of left and right Riemann sums}

\address{Département de mathématiques et de statistique, Université Laval, 2325 Rue de l'Université, Québec, QC G1V 0A6}

\email{ludovick.bouthat.1@ulaval.ca}

\author{Ludovick Bouthat}

\begin{abstract}
    Riemann sums, a classical method for approximating the definite integral of a function, have been extensively studied in the past. However, their monotonic properties, while still of great importance, particularly in approximation theory and interpolation theory, remain somewhat obscure. This paper is dedicated to proving general theorems about the monotonicity of left and right Riemann sums, a problem first raised by Fejér in 1950. We provide a much-needed review of the literature on the problem and offer several new sufficient and necessary conditions for the monotonicity of Riemann sums. Additionally, we present a new insightful proof of a fundamental theorem related to these sums using tools from the theory of majorization. The author also delves deeper into a question posed by Borwein, almost resolving it completely.
\end{abstract}



\maketitle

\vspace{-3pt}
\section{Introduction}

\subsection{Background and context}

\subsubsection*{Some Definitions}

Although they have been used sparingly throughout history, it was in 1854, in Bernhard Riemann's \emph{Habilitationsschrift} (an essay similar to a doctoral thesis), that the idea of the \emph{Riemann sums} was first put on a firm footing \cite{Riemann}. Therein, Riemann defines the \emph{Riemann integral} via the use of these sums and he introduces the concept of a \emph{Riemann integrable function}, which generalized the work of Cauchy who had already defined the integral for continuous functions. 
Recall that, given a bounded function $f:[a,b]\to\mathbb{R}$ over a partition
\vspace{-2pt}
\[
{a = x_0} < x_1 < \dots < x_{n-1} < x_n = b,\vspace{-2pt}
\]
a \emph{Riemann sum of $f$} is a finite sum of the form
\vspace{-4pt}
$$ 
\sum\limits_{k=1}^n f(x_k^*)(x_k - x_{k-1}),\vspace{-4pt}
$$
where each $x_k^*$ is chosen arbitrarily in the subinterval $[x_{k-1}, x_k]$, for each value of $k$. If, no matter which $x_k^*$'s were chosen, the difference between the Riemann sum and the Riemann integral $\int_a^b f(x) \, dx$ approaches zero as $n$ tends to infinity, then $f$ is said to be a Riemann integrable function.

Different choices of $x_k^*$ in the subinterval $[x_{k-1}, x_k]$ yield different Riemann sums of $f$. Of course, some choices are more natural than others. Hence, when $x_k^*$ is chosen to be $x_{k-1}$, the left endpoint of the subinterval $[x_{k-1}, x_k]$, it is customary to speak of the \textit{left Riemann sum of $f$}. As for when $x_k^*$ is set to be $x_k$ for all $k$, we speak of the \textit{right Riemann sum}. These are the main subject of this document.

In what follows, we shall restrict ourselves, without any loss of generality, to Riemann integrable functions defined on the interval $[0,1]$. We shall denote the \textit{left Riemann sum of $f$ with respect to the uniform partition of $[0,1]$} into $n$ intervals of equal width (of size $\frac{1}{n}$) by
$$ 
L_n(f) \,:=\, \frac{1}{n}\sum\limits_{k=0}^{n-1} f\big(\tfrac{k}{n}\big).\smallskip
$$
Similarly, we shall designate the \textit{right Riemann sum of $f$ with respect to the uniform partition of $[0,1]$} by
\begin{equation*}
    R_n(f) \,:=\, \frac{1}{n}\sum\limits_{k=1}^{n} f\big(\tfrac{k}{n}\big).\smallskip
\end{equation*}
To simplify the terminology, we simply say left (resp. right) Riemann sum to refer to $L_n$ (resp. $R_n$) in the following.

\subsubsection*{The Main Problem}

Riemann sums, as a mathematical object in itself, have been the subject of study for a long time. Naturally, as they primarily arise in the context of approximating integrals of functions, most of these efforts have focused on their asymptotic properties. For instance, it is well-known that when a function $f$ is decreasing on the interval $[0,1]$, the left Riemann sum overestimates the integral, while the right Riemann sum underestimates it. 

Nowadays, we have a very complete understanding of the asymptotic properties of Riemann sums. This, along with the rise of the Lebesgue integral beginning in 1904 \cite{Lebesgue} and the development of better numerical methods to approximate integrals, has contributed to a relative decline in the popularity of Riemann sums as an object of study in the last century (although they remain present in schools to teach students about integration). 

However, some important and interesting questions about left and right Riemann sums remain understudied. Among those is a problem which appear to have been initially proposed by Fejér in 1950 \cite{MR41284}. This problem is the main concern of this paper.

\medskip
\noindent\textbf{Main problem} (Fejér): \emph{Find good sufficient conditions on a function $f:[0,1]\to \mathbb{R}$ to ensure the monotonicity of its left and/or right Riemann sums relative to $n$.}

\subsubsection*{Some Applications}

This problem has applications in various areas, including approximation theory, sampling theory, and interpolation. For instance, in \cite{MR3277883}, the equivalence of the approximate sampling theorem and the Euler--Maclaurin summation formula, which can be used to approximate an integral by its Riemann sums, is demonstrated. In \cite{MR4214351}, the author characterizes quasi-Banach interpolation spaces for the couple $(L^p(0,\alpha),L^q(0,\alpha))$ in terms of two monotonicity properties, employing tools from majorization theory, which will be crucial for us in \Cref{sec - maj}. Additionally, in \cite{MR709348}, the authors prove the Shannon sampling theorem, two forms of Poisson's summation formula, and provide error estimates for approximations of the integral of a function $f(t)$ over segments of the real axis as a finite Riemann sum under certain asymptotic properties of $f(t)$. They even present an explicit example for the latter results using the function $f(x):= \frac{1}{1+x^2}$. Unexpectedly, the monotonicity of the Riemann sums of this same function was independently studied in detail, as we shall see in \Cref{sec - f_b1}.

Moreover, while the problem of monotonicity of Riemann sums is interesting in its own right from a theoretical point of view, it also holds applications as a means to derive various inequalities. Indeed, if the right Riemann sums of a function $f$ are monotonically increasing, then we have
\begin{equation}\label{eq - ineq_right}
    \frac{1}{n} \sum_{k=1}^n f\big(\tfrac{k}{n} \big) \,\leq\, \frac{1}{n+1} \sum_{k=1}^{n+1} f\big(\tfrac{k}{n+1} \big)
\end{equation}
for all integer $n\geq 1$. This type of inequality is useful in various contexts. For instance, in \cite{Bouthat}, it is shown that \eqref{eq - ineq_right} holds for $f(x):=\sin^p(\pi x)$ for every $0<p\leq 2$. This somewhat artificial inequality was subsequently used in \cite{Bouthat3} to address a problem at the intersection of matrix theory and metric geometry. Indeed, using this inequality, one can show that the diameter of the $n$-dimensional Birkhoff polytope, i.e., the set of $n\times n$ doubly stochastic matrices \cite{MR0020547}, with respect to the Schatten $p$-norm ($1 \leq p \leq 2)$ verifies
\vspace{-2pt}
\[
\text{diam}_{\mathcal{S}_p}(\mathcal{D}_n) \,=\, 2\!\left(\sum\limits_{k=1}^n \sin^p\!\left(\tfrac{k\pi}{n}\right)\!\right)^{\!\!\frac{1}{p}}\!.
\]

Another example is found in K.\ Jichang's paper from 1999 \cite{Kuang}. Therein, the author considered the functions $f(x)=\ln(x)$ and $f(x)=\ln\big(\frac{x}{1+x}\big)$ and, using \eqref{eq - ineq_right}, managed to provide extensions and refinements of Alzer's inequality \cite{Alzer2}. This inequality, itself a refinement of the Minc--Sathre inequality \cite{MincSathre}, asserts that if $r$ is a positive number and $n$ is a positive integer, then
\vspace{-1pt}
\begin{equation}\label{eq - Minc-Sathre}
    \frac{n}{n+1} \,\leq\, \left( \frac{\frac{1}{n}\sum_{k=1}^n k^r}{\frac{1}{n+1}\sum_{k=1}^{n+1} k^r}\right)^{\!\!\frac{1}{r}} \!\leq\, \frac{\sqrt[n]{n!}}{\sqrt[n+1]{(n+1)!}},
\end{equation}
\vspace{-1pt}
where both bounds are sharp.

\subsubsection*{What Remains to be Done}

In recent years, the problem of the monotonicity of left and right Riemann sums has experienced a resurgence in interest, thanks to the work of Kyrezi \cite{Kyrezi}, Szilárd \cite{Szilard}, Borwein, Borwein \& Sims, and Bouthat, Mashreghi \& Morneau-Guérin \cite{Bouthat}. These recent contributions build upon the foundations laid as far back as 70 years ago by Féjer, Szegö \& Turán \cite{MR41284,MR137818}, van Lint \cite{MR2408409}, Qi \cite{Qi1,Qi2,Chen}, Guo \cite{Qi2}, Chen, Cerone \& Dragomir \cite{Chen}, Jichang \cite{Kuang}, and Bennett \& Jameson \cite{Bennett}. However, the problem remains challenging, and there are still very few known sufficient conditions. This difficulty, as we shall see in \Cref{sec - maj}, has contributed to the scarcity of such conditions in the existing literature. Surprisingly, despite its natural appeal, there are very few papers dedicated to Fejér's problem. Therefore, in this paper, we aim to consolidate all the relevant research on the subject, while also introducing new results. Our objective is to formally present this problem to the mathematical community.

\medskip
\subsection{Survey of previous work}\label{sec - previous}

\subsubsection*{Early work (1950-2000) -- A fundamental theorem on left and right Riemann sums and a generalization}

Inequalities of the form \eqref{eq - ineq_right} (or very similar ones) were used by mathematicians for a long time for specific choices of $f$. However, due to their ubiquity and the lack of coherent terminology for these kinds of inequalities, it is impossible to recount all of these instances. 

Hence, let us focus in this section on general work on the monotonicity of left and right Riemann sums. The first instance of such work appears to be in 1950, when Fejér initially proposed the problem of the monotonicity of Riemann sums in \cite{MR41284}. Therein, the author gave without proofs necessary and sufficient conditions for certain Riemann sums of certain functions $f(t)$ to converge monotonically to $\int_0^1 f(t)dt$. Formal proofs seems to only have come 11 years later in \cite{MR137818} when G.\ Szeg\H{o} \& P.\ Tur\'{a}n showed, among other things, the following theorem.

\begin{thm}\textup{\cite{MR41284,MR137818}}\label{thm - Kuang1}
    Let $f:[0,1]\to \mathbb{R}$.
    \begin{enumerate}
        \item If $f$ is increasing and convex (or concave) on $[0,1]$, then $L_n(f)$ is increasing and $R_n(f)$ is decreasing; 
        \item If $f$ is decreasing and convex (or concave) on $[0,1]$, then $L_n(f)$ is decreasing and $R_n(f)$ is increasing.
    \end{enumerate}
\end{thm}

\begin{rem}
\vspace{-1.5pt}
    Restricting ourselves to the functions $R_{2^n}(f)$ and $L_{2^n}(f)$, it is easy to see geometrically that it suffices for $f$ to be increasing for $R_{2^n}(f)$ to be decreasing and $L_{2^n}(f)$ to be increasing (the opposite being true if $f$ is decreasing). However, this is no longer true in the case of a general $n$. To see this, it suffices to consider the function $f(x):=\mathbf{1}_{[\frac{1}{2},1]}(x)$.
    \vspace{-5.5pt}
\end{rem}

\Cref{thm - Kuang1} is in many ways the fundamental theorem at the basis of our problem. Note that, due to its simplicity, the theorem admits several different formulations. These divergence in terminology has hindered the clear identification of the problem and, as a result, the literature review. Hence why this theorem has been independently proven on at least five distinct occasions over the years. In 1975, J.\ H.\ van Lint used the theorem to give a solution to J.\ van de Lune's \emph{Problem 399} \cite[p.\,565]{MR2408409}. In 1999, K.\ Jichang used the result to refine Minc--Sathre's inequality \eqref{eq - Minc-Sathre}. Subsequently, it was proven at least two more times in 2010 by I.\ Kyrezi \cite{Kyrezi} and in 2012 by A.\ Szilárd \cite{Szilard}. 

Moreover, \Cref{thm - Kuang1} was also slightly generalized on different occasions, notably by F.\ Qi in 2000 \cite{Qi1} and by I.\ Gavrea in 2008 \cite{Gavrea}. We present in the following theorem the result of the first author. Note that it was initially shown in a somewhat less general context, but we have chosen to generalize it here.

\vspace{-4pt}
\begin{thm}\textup{\cite{Qi1}}
    Let $f:[0,1]\to \mathbb{R}$ and $k$ be a nonnegative integer.
    \begin{enumerate}
    \itemsep0.15em
        \item If $f$ is increasing and convex (or concave) on $[0,1]$, then the sequence $\frac{1}{n} \sum_{i=k+1}^{n+k} f(\frac{i}{n+k})$ is decreasing relative to both $n$ and $k$ and the sequence $\frac{1}{n} \sum_{i=0}^{n-1} f(\frac{i}{n+k})$ is increasing relative to both $n$ and $k$;
        \item If $f$ is decreasing and convex (or concave) on $[0,1]$, then the sequence $\frac{1}{n} \sum_{i=k+1}^{n+k} f(\frac{i}{n+k})$ is increasing relative to both $n$ and $k$ and the sequence $\frac{1}{n} \sum_{i=0}^{n-1} f(\frac{i}{n+k})$ is decreasing relative to both $n$ and $k$.
    \end{enumerate}
    \vspace{-5pt}
\end{thm}

\subsubsection*{Modern work (2000-2015) -- More generalizations on the monotonicity of Riemann sums}

This generalization was then further improved by Chen, Qi, Cerone \& Dragomir in 2003 \cite{Chen} to obtain two results, which were themselves slightly improved in 2006 by Qi \& Guo \cite{Qi2}. We combine the contributions of both papers in the following theorem.

\vspace{-4pt}
\begin{thm}\textup{\cite{Qi2,Chen}}\label{thm - weird}
    Let $f:[0,1]\to\mathbb{R}$ be an increasing function and $\varphi:(0,\infty)\to\mathbb{R}$ a positive increasing function. Define $a_0=0$, $a_k:=\varphi(k)$ $(k\geq 1)$, $A_k:=1-\frac{a_k}{a_{k+1}}$ and $B_k:=\frac{a_{k+1}}{a_k}-1$.
    \begin{enumerate}
    \itemsep0.15em
        \item If $f$ is convex (resp. concave) on $[0,1]$ and if the sequence $kA_k$ (resp. $kB_k$) is increasing, then $ \frac{1}{n} \sum_{k=1}^n f\big(\frac{a_k}{a_n}\big)$ is decreasing relative to $n$.
        \item If $f$ is convex (resp. concave) on $[0,1]$ and if the sequence $kB_k$ (resp. $kA_k$) is decreasing, then $ \frac{1}{n} \sum_{k=0}^{n-1} f\big(\frac{a_k}{a_n}\big)$ is increasing relative to $n$.
    \end{enumerate}
    If, moreover, $f$ is nonnegative, then
    \begin{enumerate}
    \itemsep0.15em
        \item[3.] If $f$ is convex (or concave) on $[0,1]$, if $\varphi$ is convex and if the sequence $a_kA_k$ is increasing, then $ \frac{1}{a_n} \sum_{k=1}^n f\big(\frac{a_k}{a_n}\big)$ is decreasing relative to $n$;
        \item[4.] If $f$ is convex (resp. concave) on $[0,1]$ and if the sequence $ a_{k+1}B_k $ (resp. $a_kA_k$) is decreasing, then $ \frac{1}{a_n} \sum_{k=0}^{n-1} f\big(\frac{a_k}{a_n}\big)$ is increasing relative to $n$.
    \end{enumerate}
    \vspace{-4pt}
\end{thm}

Obviously, by considering $-f$ instead of $f$, similar statements can be obtained for decreasing functions in the above. However, due to the somewhat cumbersome nature of the result, we omit these cases. 

\smallskip

Meanwhile, G.\ Bennett and G.\ Jameson \cite{Bennett} published a paper in 2000 in which the authors study the monotonicity of various averages of a convex or concave function at $n$ equally spaced points. In particular, they proved statements concerning
\medskip
\begin{enumerate}
    \itemsep0.4em
    \item the \emph{central Riemann sum} $C_n(f) := \frac{1}{n-1} \sum_{k=1}^{n-1} f\big( \tfrac{k}{n}\big)$,
    \item the \emph{bilateral Riemann sum} $B_n(f) := \frac{1}{n+1} \sum_{k=0}^n f\big( \tfrac{k}{n}\big)$,
    \item the \emph{mid-point Riemann sum} $M_n(f):= \frac{1}{n} \sum_{k=1}^n f\big( \frac{2k-1}{2n} \big)$,
    \item the \emph{trapezium Riemann sum} $T_n(f):=\frac{f(0)+f(1)}{2n}+ \frac{n-1}{n} C_n(f)$.
\end{enumerate} 
\medskip
Note that none of these Riemann sums can be derived by choosing appropriate values for $a_k$ in \Cref{thm - weird}. Consequently, the following theorem is entirely independent of the aforementioned one.

\begin{thm}\textup{\cite{Bennett}}\label{thm - bilateral}
    Let $f:[0,1]\to\mathbb{R}$.
    \begin{enumerate}
        \item If $f$ is convex, then $C_n(f)$ is increasing and $B_n(f)$ is decreasing;
        \item If $f$ is concave, then $C_n(f)$ is decreasing and $B_n(f)$ is increasing.
    \end{enumerate}
    Moreover, if $f'$ is either convex or concave and
    \begin{enumerate}
        \item If $f$ is convex, then $M_n(f)$ is increasing and $T_n(f)$ is decreasing;
        \item If $f$ is concave, then $M_n(f)$ is decreasing and $T_n(f)$ is increasing.
    \end{enumerate}
\end{thm}

However, while the paper presents interesting ideas and insights, it also contains minor errors. Notably, the authors asserted that \emph{if $f$ is convex or concave on $[0,1]$ (with no conditions on the monotonicity of the function), then $L_n(f)$ is increasing and $R_n(f)$ is decreasing}, which contradicts \Cref{thm - Kuang1}. To illustrate, one can consider the function $f(x)=1-x$ as a counterexample. It is worth mentioning that J.\ Rooin and H.\ Dehghan \cite{Rooin} also proved parts of the above theorem again in 2015, providing new refinements of the Hermite--Hadamard inequality and the classical Alzer's inequality in the process.
\medskip

\subsubsection*{Borwein's work (2015-2023) -- New improvements to the fundamental theorem}

Let us now go back to results concerning the monotonicity of $R_n$ and $L_n$. As we have seen, in the 65 years following the introduction of \Cref{thm - Kuang1} in the literature, the result was generalized in various directions. However, it was not directly improved as no new sufficient conditions were given to ensure the monotonicity of the left and right Riemann sums of $f$. 

It was only recently, in 2020, that D.\ Borwein, J.\ M.\ Borwein and B.\ Sims improved on the theorem by giving several new sufficient conditions to guarantee that the left and right Riemann sums are monotonic in nature \cite{Borwein}. We present these results here, starting with the following theorem.
\begin{thm}\textup{\cite{Borwein}}
    Let $f:[0,1]\to\mathbb{R}$ and let $c\in[0,1]$.
    \begin{enumerate}
        \item If $f$ is convex on the interval $[0,c]$, concave on $[c,1]$ and decreasing on $[0,1]$, then $L_n(f)$ is decreasing and $R_n(f)$ is increasing;
        \item If $f$ is concave on the interval $[0,c]$, convex on $[c,1]$ and increasing on $[0,1]$, then $L_n(f)$ is increasing and $R_n(f)$ is decreasing.
    \end{enumerate}
\end{thm}

The following theorem is useful since it does not depend on the monotonicity of the function $f$. However, it alone is insufficient to demonstrate the monotonicity of the Riemann sums of $f$, as an additional factor of $n^{-1}$ must be included.

\begin{thm}\textup{\cite{Borwein}}\label{thm - concave+}
    Let $f:[0,1]\to\mathbb{R}$ and let $c\in[0,1]$.
    \begin{enumerate}
        \item If $\hfill$the $\hfill$function$\hfill$ $f$ $\hfill$is $\hfill$concave $\hfill$on$\hfill$ $[0,1]$ $\hfill$with $\hfill$a $\hfill$maximum $\hfill$of$\hfill$ $f(c)$, $\hfill$then$\hfill$ the $\hfill$function $R_n(f)-\frac{f(c)-f(0)}{n}=L_n(f)-\frac{f(c)-f(1)}{n}$ is increasing;
        \item If $\hfill$the $\hfill$function$\hfill$ $f$ $\hfill$is $\hfill$convex$\hfill$ on$\hfill$ $[0,1]$ $\hfill$with $\hfill$a minimum$\hfill$ of$\hfill$ $f(c)$, $\hfill$then $\hfill$the $\hfill$ function \newline $R_n(f)-\frac{f(c)-f(0)}{n}=L_n(f)-\frac{f(c)-f(1)}{n}$ is decreasing.
    \end{enumerate}
\end{thm}

In the paper of Borwein \emph{et al.}, the authors introduce the concept of symmetrization (with respect to $x=\frac{1}{2}$) of the function $f$, defined by
\begin{equation}
    \mathcal{F}(x) \,:=\, \mathcal{F}_{1/2}(x) \,=\, \frac{f(x)+f(1-x)}{2}.
\end{equation}
The symmetrization of a function, as its name suggests, is symmetrical with respect to $x=\frac{1}{2}$ and satisfies the property $R_n(\mathcal{F})=L_n(\mathcal{F})$. It is easily verified that the symmetrization of a convex (resp. concave) function is once again convex (resp. concave). However, it is also possible that the symmetrization of a function that is neither convex nor concave is itself convex or concave. It is this property in particular that makes symmetrization such an interesting tool.

\smallskip
The natural analog to the left and right Riemann sums of $f$ for the symmetrization $\mathcal{F}$ is the \emph{symmetric Riemann sum}
\begin{equation}\label{eq - somme_sym}
    \lambda_n(f) \,:=\, \frac{1}{n} \sum_{k=0}^n f\big( \tfrac{k}{n}\big) - \frac{f\big(\frac{1}{2}\big)}{n},
\end{equation}
introduced by Borwein \emph{et al.} \cite{Borwein}. For this function, the authors showed the following elegant result.

\begin{thm}\textup{\cite{Borwein}}\label{thm - sym}
    Let $f:[0,1]\to\mathbb{R}$ be a symmetrical function with respect to $x=\frac{1}{2}$.
    \begin{enumerate}
        \item If $f$ is concave, then $\lambda_n(f)$ is increasing;
        \item If $f$ is convex, then $\lambda_n(f)$ is decreasing.
    \end{enumerate}
\end{thm}

Using \Cref{thm - sym}, the authors were subsequently able to show the following powerful result.

\begin{thm}\textup{\cite{Borwein}}\label{thm - Borwein}
    Let $f:[0,1]\to\mathbb{R}$.
    \begin{enumerate}
    \itemsep0.08em
        \item If $f$ has a concave symmetrization and verifies $f(0)\geq f\big(\tfrac{1}{2}\big)$, then $R_n(f)$ is increasing;
        \item If $f$ has a concave symmetrization and verifies $f(1)\geq f\big(\tfrac{1}{2}\big)$, then $L_n(f)$ is increasing;
        \item If $f$ has a convex symmetrization and verifies $f(0)\leq f\big(\tfrac{1}{2}\big)$, then $R_n(f)$ is decreasing;
        \item If $f$ has a convex symmetrization and verifies $f(1)\leq f\big(\tfrac{1}{2}\big)$, then $L_n(f)$ is decreasing.
    \end{enumerate}
\end{thm}

\subsection{The function \texorpdfstring{$f_b$}{fb}}\label{sec - f_b1}

In 2012, in an attempt to show the non-triviality of the question of the monotonicity of the left and right Riemann sums of a function, Szilárd \cite{Szilard} published a vulgarization paper in which he looked at the monotonicity of the left and right Riemann sums of the function $f_0(x)=\frac{1}{1+x^2}$. Therein, he shows that $L_n(f)$ is monotonically decreasing and that $R_n(f)$ is monotonically increasing. However, his proof was flawed.

Despite this small blunder, Szilárd's paper caught the attention of David Borwein and his collaborators who, in 2015, provided a rectified proof that $L_n(f)$ is monotonically decreasing \cite{Borwein}. However, they were unable to show that $R_n(f)$ is monotonically increasing. Nonetheless, they also proposed to generalize this problem by studying the monotonicity of the Riemann sums of the function
\vspace{-4pt}
\begin{equation}
    f_b(x) \,:=\, \frac{1}{1-bx+x^2}, \qquad (|b|<2).
    \vspace{-4pt}
\end{equation}

Using the techniques developed in their paper, they showed that $L_n(f_b)$ is monotonically decreasing for $b\in[-2,\alpha^-]$, where $\alpha^- \approx -0. 8794$ is the only negative root of the polynomial $b^3-3b^2+3$, and that $R_n(f_b)$ is monotonically increasing for $b\in\big[\frac{1-\sqrt{13}}{3},\frac{1}{2}\big]$.

In 2022, the present author, along with J.\ Mashreghi and F.\ Morneau-Guérin, tried to further analyze the function $f_b$ \cite{Bouthat}. We first noted that it is possible to consider Borwein's question for any $b<2$ (instead of $|b|<2$) since the question is also well defined if $b\leq-2$. Note, however, that $f_b$ has a singularity in $[0,1]$ if $b> 2$ and the left and right Riemann sums of $f_b$ become chaotic (the case $b=2$ is treated in \Cref{rem - b=2}).

Secondly, we were able to extend the results of Borwein, Borwein \& Sims by considering a well-chosen function $g$ and noting that $R_n(f)=R_n(f-g)+R_n(g)$ (likewise for the left Riemann sum). In particular, we were able to answer Szilárd's initial question by showing that $R_n(f_0)$ is indeed monotonically increasing. We then showed that $L_n(f_b)$ is monotonically decreasing for all $b\in(-\infty,\alpha)$, where $\alpha\approx 0.493862$, and we finally proved that $R_n(f)$ is monotonically increasing for all $b\in(-\infty,1]$.

\smallskip

\noindent However, our paper contains two small errors:
\begin{enumerate}
    \item The proof that $R_n(f)$ is monotonically increasing for all $b\in (-\frac{1}{2},1]$ is incorrect;
    \item We stated that numerically, $L_n(f_b)$ appeared to be no longer monotonic for $b>\frac{1}{2}$ and $R_n(f_b)$ appeared to be no longer monotonic for $b>1$, and thus that our results were optimal, or very close to being optimal. However, a more detailed investigation suggests that $L_n(f_b)$ is monotonically increasing for $b\in[1,2)$ and that $R_n(f_b)$ is monotonically decreasing for $b\in\big[\frac{3}{2},2\big)$.
\end{enumerate}
Our first (and most significant) error stems from the fact that the right Riemann sum of the function $k_{a,c,d}$, defined on page 10 of \cite{Bouthat}, is not decreasing as stated and thus, \cite[Theorem 5]{Bouthat} does not apply to this function.

\subsection{Outline of the paper}

The core of the paper will be presented in two main sections, themselves separated in different subsections. 

The first one will be dedicated to proving general theorems about the monotonicity of left and right Riemann sums. In particular, we first give a new proof of \Cref{thm - Kuang1} using tools from the theory of majorization. Doing so, we shine a new light on the problem, giving insight on \emph{why} this problem is so difficult. Then, we improve on \Cref{thm - Borwein} by \emph{completing} it in a very natural way. This will in turn allow us to improve our knowledge of the Riemann sums of $f_b$. 
Moreover, to allow us to easily find interesting examples that help us understand the problem, we also give a simple characterization of the polynomials of degree at most 3 whose Riemann sums are monotonic. 
We then exploit the properties of Fourier series to obtain some new interesting sufficient conditions. We also gives non-trivial examples of functions for which the new sufficient conditions allows us to deduce the monotonicity of its Riemann sums while the previously known ones did not. Finally, we complete the first section by giving necessary conditions for the monotonicity of the left and right Riemann sums.

The second part of the paper is dedicated to the left and right Riemann sums of $f_b$. In particular, we show that $L_n(f_b)$ is decreasing for all $b\in\big(-\infty, \frac{1}{2}\big]$ and increasing for all $b\in[1,\beta^+]$, while $R_n(f_b)$ is increasing for all $b\in(-\infty, 1]$ and decreasing for all $b\in\big[\frac{3}{2},\beta^+\big]$, where $\beta^+:=\frac{3+\sqrt{13}}{4}$. To do so, we will, among other things, show that the monotonicity of $L_n(f_b)$ and $R_n(f_b)$ when $b\in \big[\frac{1}{2},1.347296\big]$ follows directly from the one of $L_n(f_1)=R_n(f_1)$. Hence, we finally show that the latter is indeed monotonic by using a myriads of inequalities and tools, like the Laplace transform and the Residue theorem for sums.

\section{New general results}

\subsection{A new enlightening refinement of \texorpdfstring{\Cref{thm - Kuang1}}{Theorem 1.1}}\label{sec - maj}

It is easy to argue that \Cref{thm - Kuang1} is in some way the result at the heart of the study of the monotonicity of Riemann sums. Indeed, there have been partial results of this type for a long time and the theorem has in fact been proven independently on at least five occasions (see \cite{Kuang,Szilard,Kyrezi,MR2408409,MR137818,MR41284}). In this section, we propose a new fundamental proof, while not the most elementary, is certainly the most revealing. Before doing so,, we first need to establish a few preliminary results from the theory of majorization.

Let $x,y \in \mathbb{R}^n$. We say that $x$ is \emph{majorized} by $y$, and we write $x\prec y$, if 
\vspace{8pt}
\begin{enumerate}
    \itemsep0.7em
    \item $\sum_{j=1}^k x_j^{\uparrow} \geq \sum_{j=1}^k y_j^{\uparrow}$ for $k=1,\dots,n-1$,
    \item $\sum_{j=1}^k x_j^{\uparrow} = \sum_{j=1}^k y_j^{\uparrow}$ for $k=n$, \label{eq - maj_2}
\end{enumerate}
\vspace{8pt}
where $x^{\uparrow} = (x_1^{\uparrow},x_2^{\uparrow},\dots,x_n^{\uparrow})$ is the vector $x$ arranged in increasing order. If condition~\ref{eq - maj_2} is dropped altogether, we say that $y$ \emph{weakly supermajorize} $x$ and we write $x\prec^w y$. Intuitively, majorization says that $x$ is more \emph{evenly distributed} than $y$. Hence, this concept is naturally used in econometric and physics, among other things. For more details on majorization, see \cite{MarshallOlkinBarry2011}.

Probably the most important result in the theory of majorization is a series of equivalent statements characterizing the majorization between two vectors $x$ and $y$. The complete result can be found in \cite[A.3]{MarshallOlkinBarry2011} but here, let us only mention the equivalence of interest to us. This precise statement is often refered to as the \emph{second theorem of Hardy, Littlewood and Pólya} and states that if $x,y\in\mathbb{R}^n$, then $\sum_{j=1}^n \phi(x_j) \leq \sum_{j=1}^n \phi(y_j)$ for all convex functions $\phi$ if and only if $x\prec y$. 
Moreover, there is also a direct analog to this result in the case of weak supermajorization:

\medskip
{\itshape
\noindent Let $x,y\in\mathbb{R}^n$. Then 
\begin{equation}\label{eq - supermaj}
    \sum_{j=1}^n \phi(x_j) \,\leq\, \sum_{j=1}^n \phi(y_j)
\end{equation}
for all decreasing convex functions $\phi$ if and only if $x\prec^w y$.
}
\medskip

This result will allow us to give an alternate proof of \Cref{thm - Kuang1}. However, let us first show the following lemma.

\begin{lem}\label{lem - maj}
    Consider the following two vectors of length $n(n+1)$ :
    \begin{gather*}
        x \,:=\, \Big(\underbrace{\tfrac{1}{n},\dots,\tfrac{1}{n}}_{n+1~\text{times}},\underbrace{\tfrac{2}{n},\dots,\tfrac{2}{n}}_{n+1~\text{times}},\dots\dots,\!\!\underbrace{1,\dots,1}_{n+1~\text{times}}\!\!\Big),\\[-4pt]
        \text{and}\qquad\qquad\qquad\qquad\qquad\qquad\qquad\qquad\qquad\qquad\qquad\qquad\qquad\qquad\qquad\qquad\qquad\qquad\qquad \\[-3pt]
        y \,:=\, \Big(\underbrace{\tfrac{1}{n+1},\dots,\tfrac{1}{n+1}}_{n~\text{times}},\underbrace{\tfrac{2}{n+1},\dots,\tfrac{2}{n+1}}_{n~\text{times}},\dots\dots,\underbrace{1,\dots,1}_{n~\text{times}}\Big).
        \vspace{-5pt}
    \end{gather*}
    Then $x \prec^w y$.
    \vspace{-5pt}
\end{lem}
\begin{proof}
    It is plain that the desired result is directly obtained if the following holds true :
    \begin{align}\label{eq - maj}
        \sum_{j=1}^k x_j = \sum_{j=1}^k x_j^{\uparrow} \,&\leq\, \sum_{j=1}^k y_j^{\uparrow} = \sum_{j=1}^k y_j, \qquad k=1,\dots,n(n+1).
    \end{align}
    Let us consider two cases.

    \medskip
    \noindent {\bf Case 1:\, $\bm{m(n+1) \!\leq\! k \!\leq\! (m+1)n}$,\, $\bm{0 \!\leq\! m \!\leq\! n}$.}
    \vspace{2pt}

    \noindent In this case, it is a matter of direct computation to verify that on the one hand,
    \begin{align*}
        \sum_{j=1}^{k}x_{j} \,&=\, \sum_{j=1}^{n+1}x_{j}+\sum_{j=n+2}^{2\left(n+1\right)}x_{j}+\cdots+\!\!\sum_{j=m\left(n+1\right)+1}^{k} \!\!x_{j} \\[-1pt]
        &=\, \sum_{j=1}^{n+1}\frac{1}{n}+\sum_{j=n+2}^{2\left(n+1\right)}\frac{2}{n}+\cdots+\!\!\sum_{j=m\left(n+1\right)+1}^{k}\!\frac{m+1}{n} \\[-1pt]
        &=\,\frac{n+1}{n}+2\frac{n+1}{n}+\cdots+m\frac{n+1}{n}+\frac{m+1}{n}\left(k-m\left(n+1\right)\right) \\[-1pt]
        &=\, \frac{m+1}{n}\left(k-\frac{m\left(n+1\right)}{2}\right),
    \end{align*}
    and on the other
    \begin{align*}
        \sum_{j=1}^{k}y_{j} \,&=\, \sum_{j=1}^{n}y_{j}+\sum_{j=n+1}^{2n}y_{j}+\cdots+\sum_{j=mn+1}^{k}y_{j} \\[-1pt]
        &=\, \sum_{j=1}^{n}\frac{1}{n+1}+\sum_{j=n+1}^{2n}\frac{2}{n+1}+\cdots+\sum_{j=mn+1}^{k}\frac{m+1}{n+1} \\[-1pt]
        &=\, \frac{n}{n+1}+\frac{2n}{n+1}+\cdots+\frac{mn}{n+1}+\frac{(m+1)\left(k-mn\right)}{n+1} \\[-1pt]
        &=\, \frac{m+1}{n+1}\left(k-\frac{mn}{2}\right).
    \end{align*}
    Now, 
    \begin{align*}
        \frac{m+1}{n}\left(k-\frac{m\left(n+1\right)}{2}\right) = \sum_{j=1}^{k}x_{j} \,\ge\, \sum_{j=1}^{k}y_{j} = \frac{m+1}{n+1}\left(k-\frac{mn}{2}\right)
    \end{align*}
    is trivially true if $m=0$ and is verified if and only if $k\ge m\left(n+\frac{1}{2}\right)$ when $m\geq 1$. However, by hypothesis $k\geq m (n+1) \geq m\left(n+\frac{1}{2}\right) $ and thus, \eqref{eq - maj} is satisfied in this case.

    \medskip\smallskip
    \noindent {\bf Case 2:\, $\bm{mn \!\leq\! k \!\leq\! m(n+1)}$,\, $\bm{1 \!\leq\! m \!\leq\! n}$.}
    \vspace{2pt}

    \noindent Once again, it is a matter of direct computation to verify that on the one hand,
    \begin{align*}
        \sum_{j=1}^{k}x_{j} \,&=\, \sum_{j=1}^{n+1}x_{j}+\sum_{j=n+2}^{2\left(n+1\right)}x_{j}+\cdots+\sum_{j=\left(m-1\right)\left(n+1\right)+1}^{k}x_{j} \\[-1pt]
        &=\, \sum_{j=1}^{n+1}\frac{1}{n}+\sum_{j=n+2}^{2\left(n+1\right)}\frac{2}{n}+\cdots+\sum_{j=\left(m-1\right)\left(n+1\right)+1}^{k}\frac{m}{n} \\[-1pt]
        &=\, \frac{n+1}{n}+2\frac{n+1}{n}+\cdots+\left(m-1\right)\frac{n+1}{n}+\frac{m}{n}\left(k-\left(m-1\right)\left(n+1\right)\right) \\[-1pt]
        &=\, \frac{m}{n}\left(k-\frac{\left(m-1\right)\left(n+1\right)}{2}\right),
    \end{align*}
    and on the other
    \begin{align*}
        \sum_{j=1}^{k}y_{j} \,&=\, \sum_{j=1}^{n}y_{j}+\sum_{j=n+1}^{2n}y_{j}+\cdots+\sum_{j=mn+1}^{k}y_{j} \\[-1pt]
        &=\, \sum_{j=1}^{n}\frac{1}{n+1}+\sum_{j=n+1}^{2n}\frac{2}{n+1}+\cdots+\sum_{j=mn+1}^{k}\frac{m+1}{n+1} \\[-1pt]
        &=\, \frac{n}{n+1}+\frac{2n}{n+1}+\cdots+\frac{mn}{n+1}+\frac{\left(m+1\right)\left(k-mn\right)}{n+1} \\[-1pt]
        &= \,\frac{m+1}{n+1}\left(k-\frac{mn}{2}\right).
    \end{align*}
    Now, 
    \begin{align*}
        \frac{m}{n}\left(k-\frac{\left(m-1\right)\left(n+1\right)}{2}\right) = \sum_{j=1}^{k}x_{j} \,\ge\, \sum_{j=1}^{k}y_{j} = \frac{m+1}{n+1}\left(k-\frac{mn}{2}\right)
    \end{align*}
    is clearly true if $m=n$ and is verified if and only if $k\le\frac{m}{n-m}\big(n^{2}-(n+\frac{1}{2})(m-1)\big)$ when $m< n$. However, by hypothesis $k\leq m(n+1) $ and we have $n+1\le\frac{1}{n-m}\big(n^{2}-(n+\frac{1}{2})(m-1)\big)$ if and only if $m\ge-1$, which is clearly true. Hence, we have
    \[
    k \,\leq\, m(n+1) \,\leq\, \frac{m}{n-m}\big(n^{2}-(n+\tfrac{1}{2})(m-1)\big)
    \]
    and thus, \eqref{eq - maj} is satisfied for every $1\leq k \leq n(n+1)$.
\end{proof}

Using \Cref{lem - maj}, we can now give an alternative proof of \Cref{thm - Kuang1}.

\begin{proof}[New proof of \Cref{thm - Kuang1}]
    Fix $n\in\mathbb{N}$ and define the function $F:[0,1]^{n(n+1)}\to \mathbb{R}$ by 
    \[
    F(z) \,:=\, \sum_{k=1}^{n(n+1)} f(z_k).
    \]
    Moreover, define the vectors $x, y \in [0,1]^{n(n+1)}$ as in \Cref{lem - maj}. It is then easily seen that
    \begin{equation*}
        \frac{1}{n} \sum_{k=1}^n f\big(\tfrac{k}{n}\big) \,\leq\, \frac{1}{n+1} \sum_{k=1}^{n+1} f\big(\tfrac{k}{n+1}\big).
    \end{equation*}
    is in fact \emph{equivalent} to having $F(x) \leq F(y)$. Indeed, simply multiply both sides of the inequality by $n(n+1)$. Now, since $x\prec^w y$ by \Cref{lem - maj}, it directly follows from the weak supermajorization analog of the second theorem of Hardy, Littlewood and Pólya \eqref{eq - supermaj} that we have $F(x) \leq F(y)$ whenever $f$ is convex and decreasing on $[0,1]$. In other words, $R_n(f) \leq R_{n+1}(f)$ if $f$ is convex and decreasing.
\end{proof}

\begin{rem}
\vspace{-1pt}
    The above proof works as long as $x \prec^w y$. Hence, by generalizing \Cref{lem - maj}, it could also provide a new proof of \Cref{thm - weird}. In fact, it could also be used to provide an optimal generalization to \Cref{thm - Kuang1}.
\end{rem}

\begin{rem}
\vspace{-2pt}
    \eqref{eq - supermaj} is an equivalence result. Hence, not only does it allows us to give a new proof to \Cref{thm - Kuang1}, it also reveal that this result is \emph{sufficient} to characterize the supermajorization between the vectors $x$ and $y$ in \Cref{lem - maj}. Therefore, this new proof intuitively shows why providing new sufficient conditions for the monotonicity of the Riemann sums is hard, since every one of these possible results will be about a family of functions which is not \emph{natural} to this problem.
\end{rem}

\subsection{A refinement of \texorpdfstring{\Cref{thm - Borwein}}{Theorem 1.12}}

\Cref{thm - Borwein} is very useful. However, it appears somewhat incomplete. Indeed, in all of the previous cases, statements about the monotonicity of $R_n$ (resp. $L_n$) all came with a similar statement on $L_n$ (resp. $R_n$). It turns out that it is possible to complete the following result.
\vspace{-3pt}

\begin{thm}\label{thm - improvement1}
    Let $f:[0,1]\to\mathbb{R}$.
    \begin{enumerate}
    \itemsep0.03em
        \item If $f$ has a concave symmetrization and verifies $f(0)\geq f\big(\tfrac{1}{2}\big)$, then $L_n(f)$ is decreasing and $R_n(f)$ is increasing;
        \item If $f$ has a concave symmetrization and verifies $f(1)\geq f\big(\tfrac{1}{2}\big)$, then $L_n(f)$ is increasing and $R_n(f)$ is decreasing;
        \item If $f$ has a convex symmetrization and verifies $f(1)\leq f\big(\tfrac{1}{2}\big)$, then $L_n(f)$ is decreasing and $R_n(f)$ is increasing;
        \item If $f$ has a convex symmetrization and verifies $f(0)\leq f\big(\tfrac{1}{2}\big)$, then $L_n(f)$ is increasing and $R_n(f)$ is decreasing.
    \end{enumerate}
\end{thm}

\vspace{-3pt}
Some of these results were already proven by Borwein \emph{et al.} in \Cref{thm - Borwein}. However, some of them are new and need to be proved. To do so, we need to establish the following lemma.

\begin{lem}\label{lem - concave}
    Let $f:[0,1]\to\mathbb{R}$ and
    \begin{equation*}
        g(x) \,:=\, \begin{cases}
            0 \quad &\mathrm{if}~~ 0\leq x \leq \tfrac{1}{2}; \\
            2\mathcal{F}(x)-2f\big(\tfrac{1}{2}\big) \quad &\mathrm{if}~~ \tfrac{1}{2}\leq x \leq 1.
        \end{cases}
    \end{equation*}
    Then $g$ is concave and decreasing on $[0,1]$ if and only if the symmetrization of $f$ is concave.
\end{lem}
\begin{proof}
    Assume that $g$ is concave and decreasing on $[0,1]$. Since the sum of concave functions is concave and
    \begin{equation*}
        \mathcal{F}(x) \,=\, \frac{g(x)+g(1-x)}{2}+f\big(\tfrac{1}{2}\big), 
    \end{equation*}
    it immediately follows that $\mathcal{F}(x)$ is concave.

    \smallskip

    Conversely, if $\mathcal{F}(x)$ is concave, then its restriction $\mathcal{F}\rvert_{[1/2,1]}$ is also concave. Suppose that its restriction is also not decreasing. Then there exists $\frac{1}{2}\leq a < b \leq 1$ such that $\mathcal{F}(a) < \mathcal{F}(b)$. Since $\mathcal{F}$ is concave, we must have
    \begin{equation*}
        t\mathcal{F}(1-b) + (1-t)\mathcal{F}(b) \,\leq\, \mathcal{F}(t(1-b)+(1-t)b)
    \end{equation*}
    for all $t\in[0,1]$. Choose $t\in(0,1)$ such that $t(1-b)+(1-t)b=a$ (which is possible since $1-b< a < b$); then
    \begin{equation*}
        t\mathcal{F}(1-b) + (1-t)\mathcal{F}(b) \,\leq\, \mathcal{F}(a).
    \end{equation*}
    But since $\mathcal{F}$ is symmetric about $x=1/2$, $\mathcal{F}(1-b)=\mathcal{F}(b)$ and it follows that
    \begin{equation*}
        \mathcal{F}(b) \,=\, t\mathcal{F}(1-b) + (1-t)\mathcal{F}(b) \,\leq\, \mathcal{F}(a),
    \end{equation*}
    a contradiction. Consequently, $\mathcal{F}\rvert_{[1/2,1]}=g\rvert_{[1/2,1]}$ is concave and decreasing.

    Clearly, $g\rvert_{[0,1/2]}$ is also concave and decreasing since it is constant. Hence, all that remains to show is that $g$ is concave on $[0,1]$. To do this, we need to show that for all $a,b\in[0,1]$ (with $a<b$) and all $t\in[0,1]$ we have
    \begin{equation*}
        tg(a) + (1-t)g(b) \,\leq\, g(ta+(1-t)b).
    \end{equation*}
    If $a,b\in \big[0,\frac{1}{2}\big]$ or $a,b\in \big[\frac{1}{2},1\big]$, we have already shown this to be true. Hence, let $a\in\big[0,\frac{1}{2}\big]$ and $b\in\big[\frac{1}{2},1\big]$. Then we have
    \begin{align*}
        tg(a) + (1-t)g(b) = tg\big(\tfrac{1}{2}\big) + (1-t)g(b) \leq g\big(\tfrac{t}{2}+(1-t)b\big) \leq g(ta+(1-t)b),
    \end{align*}
    the first inequality being verified by the concavity of $g$ on $\big[\frac{1}{2},1\big]$, and the second by the decreasing nature of $g$ on $[0,1]$. Hence, $g$ is concave and decreasing on $[0,1]$, which concludes the proof.
\end{proof}

We can now address the proof of \Cref{thm - improvement1}.

\begin{proof}[Proof of \Cref{thm - improvement1}]
    Points (2), (3) and (4) of the Theorem follows directly from (1) by considering the functions $f(1-x)$, $-f(1-x)$ and $-f(x)$, respectively. Hence, let us focus on the case of (1). Moreover, the fact that $R_n(f)$ is increasing in (1) is the content of \Cref{thm - Borwein}. Therefore, let us show that $L_n(f)$ decreases monotonically with respect to $n$. 
    To do this, consider the functions
    \begin{equation*}
        g(x) \,:=\, \begin{cases}
            0 \quad &\mathrm{if}~~ 0\leq x \leq \tfrac{1}{2}; \\
            2\mathcal{F}(x)-2f\big(\tfrac{1}{2}\big) \quad &\mathrm{if}~~ \tfrac{1}{2}\leq x \leq 1
        \end{cases}
    \end{equation*}
    and $h(x):=f(x)-g(x)$. By \Cref{lem - concave}, $g$ is a concave and decreasing function on $[0,1]$. Therefore, \Cref{thm - Kuang1} ensures that $L_n(g)$ is a decreasing function in $n$. Moreover, a direct computation reveals that the symmetrization of $h$ is constant and equal to $f\big(\frac{1}{2}\big)$. In particular, the symmetrization of $h$ is convex and $h\big(\frac{1}{2}\big) \geq h(1)$. Moreover, since $f(0)\geq f\big(\frac{1}{2}\big)$, we have
    \[
    h\big(\tfrac{1}{2}\big) = f\big(\tfrac{1}{2}\big) \,\geq\, f(1) - f(1)-f(0)+2f\big(\tfrac{1}{2}\big) =f(1)-g(1) = h(1).
    \]
    Hence, $\Tilde{h}(x):=-h(1-x)$ is a concave function satisfying $\Tilde{h}(0) \geq \Tilde{h}\big(\frac{1}{2}\big)$ and \Cref{thm - Borwein} ensures us that $R_n(\Tilde{h})$ is increasing and thus that $L_n(h)$ is a decreasing function.
    Therefore, since $L_n(f)=L_n(g+h)=L_n(g)+L_n(h)$, it follows that the left Riemann sum of $f$ is a monotonically decreasing function of $n$. 
\end{proof}

\subsection{A characterization for small-order polynomials}\label{sec - poly}

To better understand the problem and to illustrate its difficulty, it is helpful to have a definite result on at least one simple family of function. Since the Riemann sum of a polynomial is always a polynomial, it is possible to obtain a characterization of the monotonic Riemann sums in this particular case for various degrees of polynomials. In order to have a non-trivial case while still having an elegant result and proof, we treat the case of polynomials of order 3.

\begin{thm}\label{thm - polynome}
    Let $p$ be a polynomial of order 3. Then 
    \begin{enumerate}
        \item $L_n(p)$ is increasing if and only if $p(0) \leq \min\big\{p(1),\ p\big(\frac{1}{2}\big)\big\}$;
        \item $L_n(p)$ is decreasing if and only if $p(0)\geq \max\big\{p(1),\ p\big(\frac{1}{2}\big)\big\}$;
        \item $R_n(p)$ is increasing if and only if $p(1) \leq \min\big\{p(0),\ p\big(\frac{1}{2}\big)\big\}$;
        \item $R_n(p)$ is decreasing if and only if $p(1)\geq \max\big\{p(0),\ p\big(\frac{1}{2}\big)\big\}$.
    \end{enumerate}
\end{thm}
\begin{proof}
    Clearly, it suffices to show the assertion for case (1). $p$ is then of the form
    \[
    p(x) \,=\, 2\alpha x^{3}+\big(2p_{1}+2p_{0}-4p_{1/2}-3\alpha\big)x^{2}+\big(\alpha+4p_{1/2}-p_{1}-3p_{0}\big)x+p_{0},
    \]
    where $p_0:=p(0)$, $p_{1/2}:=p\big(\frac{1}{2}\big)$, $p_1=p(1)$ and $\alpha$ is a real number. Using Faulhalmer's formula for the sum of the $p^{\text{th}}$ power of the $n$ first integers, we find that 
    \[
    L_n(f) \,=\, \frac{(p_{1}+4p_{1/2}+p_{0})n^2+3(p_{0}-p_{1})n+2p_{1}-4p_{1/2}+2p_{0}}{6n^{2}},
    \]
    and thus
    \[
    L_{n+1}(f)-L_n(f) \,=\, -\frac{3(p_{0}-p_{1})n^{2}+(p_{1}-8p_{1/2}+7p_{0})n+2(p_{1}-2p_{1/2}+p_{0})}{6n^{2}(n+1)^{2}}.
    \]
    Hence, for $L_n(f)$ to be increasing, it is sufficient to show that 
    \[
    3\left(p_{0}-p_{1}\right)n^{2}+\big(p_{1}-8p_{1/2}+7p_{0}\big)n+2\big(p_{1}-2p_{1/2}+p_{0}\big) \,\leq\, 0
    \]
    for any integer $n\geq 1$. Equivalently (replacing $n$ by $n+1$), we want to show that 
    \begin{equation}\label{eq - polynome}
        3\left(p_{0}-p_{1}\right)n^{2}+\big(13p_{0}-5p_{1}-8p_{1/2}\big)n+12\big(p_{0}-p_{1/2}\big) \,\le\, 0
    \end{equation}
    for any $n\geq 0$. Clearly, considering $n\to\infty$ and $n=0$, we find that we must have 
    \begin{equation}\label{eq - cond}
        p_0-p_1 \,\leq\, 0 \qquad \& \qquad p_0-p_{1/2} \,\leq\, 0.
    \end{equation}
    These conditions are therefore necessary. Now, if these conditions are satisfied, we also find 
    \[
    13p_{0}-5p_{1}-8p_{1/2} \,=\, 5(p_0-p_{1})+8(p_0-p_{1/2}) \,\leq\, 0.
    \]
    Thus, all coefficients in the polynomial of \eqref{eq - polynome} are negative and it follows that \eqref{eq - polynome} is satisfied. The conditions \eqref{eq - cond} are therefore also sufficient. The conclusion follows directly by noting that \eqref{eq - cond} is equivalent to having $p(0) \leq \min\big\{p(1),\ p\big(\frac{1}{2}\big)\big\}$.
\end{proof}

\begin{rem}
\vspace{-3pt}
    The conditions in \Cref{thm - polynome} are not sufficient to ensure the monotonicity of either the left or right Riemann sums of polynomials of greater degree. For example, if $p(x) \hspace{-.3pt}=\hspace{-.3pt}x\hspace{-.3pt}-\hspace{-.3pt}7x^{3}\hspace{-.3pt}+\hspace{-.3pt}6x^{4}$, then $p(0) \hspace{-.3pt}\leq\hspace{-.3pt} \min\!\big\{p(1),\, p\big(\frac{1}{2}\big)\hspace{-1pt}\big\}$ but $L_n(p)\hspace{-.3pt}=\hspace{-.3pt}\frac{-n^{4}+5n^{2}-4}{20n^{4}}$ is decreasing in $n$. 
\end{rem}

\vspace{-4pt}
\Cref{thm - polynome} show a difficulty in the problem of the monotonicity of Riemann sums. Indeed, the theorem implies that a small perturbation of a function can destroy the monotonicity of its Riemann sums. Hence, in general, it is not possible to assert anything about the monotonicity of the Riemann sums of a certain function simply with the monotonicity of the Riemann sums of an approximation of the function, however good it may be.

\subsection{A convenient application of Fourier series}

One of main challenge in determining the monotonicity of Riemann sums is the fact that most of the time, there is no simple explicit form for $R_n(f)$ and $L_n(f)$. For instance, in \Cref{sec - poly} we are able to completely characterize the monotonic Riemann sums of polynomials of degree $\leq 3$ since in this case, the Riemann sums are simply polynomials. Hence, one strategy in the general case is to express $f$ as a combination of simpler functions for which we do have these explicit forms.

For example, in \cite{Bouthat}, the authors expresses $\sin^p(\pi x)$ as an infinite sum of functions of the type $\cos^{2m}(\pi x)$, where $m$ is a positive integer. For these functions, there exist a formula for the left and right Riemann sums which allowed the authors to show that $L_n(\sin^p(\pi x))=R_n(\sin^p(\pi x))$ is indeed monotonic.

Naturally, one could consider the Taylor expansion of the function $f$ since we know explicitly the form of the left and right Riemann sums of $x^m$ by Faulhaber's formula. However, in practice, the polynomials obtained by Faulhaber's formula are complicated and are often not convenient for this application. One family of function which is especially convenient for us is the trigonometric functions $\sin(2\pi m x)$ and $\cos(2 \pi mx)$.
\vspace{-1pt}

\begin{lem}\label{lem - sincos}
    Let $m\in \mathbb{N}$, $s_m(x):=\sin(2\pi mx)$ and $c_m(x):=\cos(2\pi mx)$. Then
    \begin{equation*}
        L_n(s_m)=R_n(s_m)=0 \qquad\&\qquad L_n(c_m)=R_n(c_m) = \begin{cases}
            1 \!&\text{if } \tfrac{m}{n} \in \mathbb{N}, \\
            0 \! &\text{if } \tfrac{m}{n} \notin \mathbb{N}.
        \end{cases}
    \end{equation*}
    \vspace{-22pt}
\end{lem}
\begin{proof}
    Let us begin with the case of $s_n$. Clearly, $L_n(s_m)=R_n(s_m)$ because $s_m(0)=s_m(1)$. Hence, $L_n(s_m)=R_n(s_m)=\frac{1}{n}\sum_{k=0}^n s_m\big(\frac{k}{n}\big)$. Moreover,
    \begin{align*}
        \frac{1}{n}\sum_{k=0}^n s_m\big(\tfrac{k}{n}\big) \,=\, \frac{1}{n}\sum_{k=0}^n s_m\big(\tfrac{n-k}{n}\big) \,=\, \frac{1}{n}\sum_{k=0}^n s_m\big(-\tfrac{k}{n}\big) \,=\, -\frac{1}{n}\sum_{k=0}^n s_m\big(\tfrac{k}{n}\big)
    \end{align*}
    and it follows that $L_n(s_m)=R_n(s_m)=0$. 
    
   As for the case of $c_m$, let us suppose that $\frac{m}{n} \in \mathbb{N}$. Then
    \begin{equation*}
        L_n(c_m) \,=\, \frac{1}{n}\sum_{k=0}^{n-1}\cos\big(\tfrac{2\pi mk}{n}\big) \,=\, \frac{1}{n}\sum_{k=0}^{n-1} 1 \,=\, 1.
    \end{equation*}
    Similarly, we also find that $R_n(c_m)=1$ if $\frac{m}{n} \in \mathbb{N}$. If $\frac{m}{n} \notin \mathbb{N}$, then
    \begin{equation*}
        L_n(c_m) \,=\, \frac{1}{n}\sum_{k=0}^{n-1}\cos\big(\tfrac{2\pi mk}{n}\big) \,=\, \frac{1}{n}\Re\!\left(\sum_{k=0}^{n-1}e^{\frac{2\pi mk}{n}i}\right) \,=\, \frac{1}{n}\Re\!\left(\frac{1-e^{2\pi im}}{1-e^{\frac{2\pi im}{n}}}\right) \,=\, 0,
    \end{equation*}
    since $1-e^{\frac{2\pi im}{n}} \neq 0$. In the same way, we also find that $R_n(c_m)=0$ if $\frac{m}{n} \notin \mathbb{N}$.
\end{proof}

In light of \Cref{lem - sincos}, it is natural to consider functions who can be expressed as a sum of sine and cosine, i.e., functions whose Fourier series on $[0,1]$ converges at each point of $[0,1]\cap \mathbb{Q}$. In this case we have the following representations of the left and right Riemann sums. For the proof of this theorem, we'll need Dirichlet's theorem which states that if a periodic function $f(x)$ is of bounded variation on a period, then its Fourier series converge at each point of the domain to $\frac{f(x^+)+f(x^-)}{2}$.

\begin{thm}\label{thm - Fourier}
    Let $f:[0,1]\to \mathbb{R}$ be a continuous function of bounded variation on $[0,1]$, and let $\Tilde{f}(x):\mathbb{R}\to\mathbb{R}$ be defined by $\Tilde{f}(x):=f(\{x\})$, where $\{x\}:=x-\lfloor x\rfloor$ is the \emph{fractional part} of $x$. Suppose that the Fourier series of $\Tilde{f}$ is
    \begin{equation*}
        \Tilde{f}(x) \,\sim\, a_{0}+\sum_{k=1}^{\infty}\left(a_{k}\cos\left(2\pi kx\right)+b_{k}\sin\left(2\pi kx\right)\right).
        \vspace{-2pt}
    \end{equation*} 
    Then
    \vspace{-2pt}
    \begin{align*}
        L_n(f) - \tfrac{f(0)-f(1)}{2n} \,&=\, R_n(f) + \tfrac{f(0)-f(1)}{2n} \,=\, a_{0}+\sum_{k=1}^{\infty}a_{nk}  \\
            &=\, \!\int_{0}^{1}\!f(u)du +\! \lim_{m\to \infty} \!2\!\int_{0}^{1}\!f(u)\frac{\sin(\pi mnu)\cos(\pi(m+1)nu)}{\sin(\pi nu)}du.
    \end{align*}
    \vspace{-10pt}
\end{thm}
\begin{proof}
    The identity $L_n(f) - \tfrac{f(0)-f(1)}{2n} = R_n(f) + \tfrac{f(0)-f(1)}{2n}$ is a matter of direct verification. Hence, let us first note that we have
    \begin{equation}\label{eq - Fourier}
    \vspace{-1pt}
        R_n(f) \,=\, \frac{1}{n} \sum_{k=1}^n f\big(\tfrac{k}{n} \big) \,=\, \frac{f(1)}{n} + \frac{1}{n} \sum_{k=1}^{n-1} \tilde{f}\big(\tfrac{k}{n} \big).
    \vspace{-.5pt}
    \end{equation}
    By Dirichlet's theorem and the continuity of $f$, we know that the Fourier series of $\Tilde{f}$ at each point $x\in(0,1)$ converges to $\Tilde{f}(x)=f(x)$. However, $\tilde{f}$ is discontinuous at each integer $m$, and in particular at the points $x=0$ and $1$. For these points, the Fourier series of $\tilde{f}$ converges to $\frac{\tilde{f}(m^+)+\tilde{f}(m^-)}{2}=\frac{\tilde{f}(0^+)+\tilde{f}(0^-)}{2}=\frac{f(0)+f(1)}{2}$. That is, we have
    \begin{equation}\label{eq - Fourier2}
        \frac{f(0)+f(1)}{2} \,=\, \sum_{j=0}^{\infty}a_{j} \,=\, a_{0}+\sum_{j=1}^{\infty}\left(a_{j}\cos\left(2\pi j\right)+b_{j}\sin\left(2\pi j\right)\right).
    \end{equation}
    
    Now, since the Fourier series of $\Tilde{f}$ converges to $f(x)$ for each $x\in(0,1)$, we can replace $\Tilde{f}$ in the right-most sum of \eqref{eq - Fourier} by its Fourier series. Doing so, we obtain
    \begin{align*}
        R_n(f) -\frac{f(1)}{n} \,&=\, \frac{1}{n}\sum_{k=1}^{n-1} f\big(\tfrac{k}{n} \big) \,=\, \frac{1}{n}\sum_{k=1}^{n-1} \bigg( a_{0}+\sum_{j=1}^{\infty}\left(a_{j}\cos\!\left(2\pi j\tfrac{k}{n}\right)+b_{j}\sin\!\left(2\pi j\tfrac{k}{n}\right)\right)\!\!\bigg).
    \end{align*}
    Adding a multiple of $\frac{1}{n}$ of \eqref{eq - Fourier2} to both sides of this equation yield
    \begin{align*}
        R_n(f) +\frac{f(0)-f(1)}{2n} \,&=\,  \frac{1}{n}\sum_{k=1}^{n} \bigg( a_{0}+\sum_{j=1}^{\infty}\left(a_{j}\cos\!\left(2\pi j\tfrac{k}{n}\right)+b_{j}\sin\!\left(2\pi j\tfrac{k}{n}\right)\right)\!\!\bigg) \\[-2pt]
        &=\, a_0 + \sum_{j=1}^{\infty} (a_j R_n(c_j)+b_jR_n(s_j)),
    \end{align*}
    where $c_j$ and $s_j$ are defined as in \Cref{lem - sincos}. Hence, an application of this lemma finally yield
    \begin{align*}
        R_n(f) +\frac{f(0)-f(1)}{2n} \,=\, a_0 + \sum_{j=1}^{\infty} a_{nj}.\\[-17pt]
    \end{align*}

    Now, the right-hand side can be reformulated in a more explicit way. Indeed, recall that we have $a_0 = \int_{0}^{1}\!\tilde{f}(x)dx = \int_{0}^{1}\!f(x)dx$ and
    \begin{equation*}
        a_n \,=\, 2\!\int_{0}^{1}\!\tilde{f}(x)\cos(2\pi nx)dx \,=\, 2\!\int_{0}^{1}\!f(x)\cos(2\pi nx)dx, \qquad n\geq 1.
    \end{equation*}
    Hence, it follows that 
    \begin{align*}
        a_0 + \sum_{j=1}^{\infty} a_{nj} \,&=\, \int_{0}^{1}f(x)dx + 2\sum_{j=1}^{\infty} \int_{0}^{1}f(x)\cos(2\pi nx)dx \\[-4pt]
        &=\, \int_{0}^{1}f(x)dx + \lim_{m\to \infty} 2\sum_{j=1}^{m} \int_{0}^{1} f(x)\cos(2\pi nx)dx \\[-2pt]
        &=\, \int_{0}^{1}f(x)dx + \lim_{m\to \infty} 2 \int_{0}^{1} f(x) \bigg(\sum_{j=1}^{m}\cos(2\pi nx) \bigg)dx \\[-2pt]
        &=\, \int_{0}^{1}f(x)dx + \lim_{m\to \infty} 2 \int_{0}^{1} f(x) \frac{\sin(\pi mnx)\cos(\pi(m+1)nx)}{\sin(\pi nx)} dx,
    \end{align*}
    where the last line is obtained by taking the real part of the geometric sum $\sum_{j=1}^{m} e^{2\pi i nx}$.
\end{proof}

In \cite{pannikov1970convergence} and \cite{petrovich1975properties}, the authors show that if a function can be represented by a trigonometric series with coefficients forming a monotonic sequence decreasing to zero, then its left Riemann sum 
converge in measure to half of the first coefficient. Under similar assumption, we show that the convergence is also monotone.
\vspace{-4pt}

\begin{cor}\label{cor - Fourier}
    Let $f:[0,1]\to \mathbb{R}$ be a continuous function of bounded variation on $[0,1]$, and let $\Tilde{f}(x):\mathbb{R}\to\mathbb{R}$ be defined by $\Tilde{f}(x):=f(\{x\})$, where $\{x\}:=x-\lfloor x\rfloor$ is the \emph{fractional part} of $x$. Suppose that the Fourier series of $\Tilde{f}$ is
    \begin{equation*}
        \Tilde{f}(x) \,\sim\, a_{0}+\sum_{k=1}^{\infty}\left(a_{k}\cos\left(2\pi kx\right)+b_{k}\sin\left(2\pi kx\right)\right).
    \end{equation*} 
    \begin{enumerate}
    \itemsep0em
        \item If $f(1)\geq f(0)$ and the coefficients $a_n$ are decreasing for $n\geq 1$, then $R_n(f)$ is monotonically decreasing relative to $n$;
        \item If $f(0)\geq f(1)$ and the coefficients $a_n$ are decreasing for $n\geq 1$, then $L_n(f)$ is monotonically decreasing relative to $n$;
        \item If $f(0)\geq f(1)$ and the coefficients $a_n$ are increasing for $n\geq 1$, then $R_n(f)$ is monotonically increasing relative to $n$;
        \item If $f(1)\geq f(0)$ and the coefficients $a_n$ are increasing for $n\geq 1$, then $L_n(f)$ is monotonically increasing relative to $n$.
    \end{enumerate}
\end{cor}
\begin{proof}
    By considering $-f(x)$, $f(1-x)$ and $-f(1-x)$, it is sufficient to only prove (1). By \Cref{thm - Fourier}, we have
    \[
    R_n(f) \,=\, a_{0}+\sum_{k=1}^{\infty}a_{nk} + \tfrac{f(1)-f(0)}{2n}.
    \]
    Since $f(1)\geq f(0)$, $\tfrac{f(1)-f(0)}{2n}$ is a decreasing function. Moreover, since the coefficients $a_n$ are decreasing, we have $\sum_{k=1}^{\infty}a_{(n+1)k} \leq \sum_{k=1}^{\infty}a_{nk}$ and it follows that $R_n(f)$ is decreasing.
\end{proof}

\begin{example}
    Consider the function $f(x):=e^{\cos\left(2\pi x\right)}\cos\left(\sin\left(2\pi x\right)\right)$. The function is symmetric and is neither convex nor concave. It is easy to verify that none of the previously mentioned theorems is able to show that the Riemann sums of $f$ are monotonic. However, $f$ is clearly continuous and of bounded variation on $[0,1]$ and the Fourier series of $\Tilde{f}$ is 
    \[
    \tilde{f}(x) \,\sim\, \sum_{k=0}^{\infty}\frac{\cos\left(2\pi kx\right)}{k!}.
    \]
    Hence, its Fourier coefficients $a_n$ are decreasing. Moreover, $f(0)=f(1)$ and thus, \Cref{cor - Fourier} ensures that $R_n(f)=L_n(f)$ is monotonically decreasing relative to $n$.
\end{example}

\begin{rem}
    Similarly, the Fourier transform and the Laplace transform can be used in a similar manner if $f$ is extended in a \emph{nice} way to be defined on $\mathbb{R}$ instead of $[0,1]$. In \Cref{sec - f1} for instance, the Laplace transform is used to show the monotonicity of the left and right Riemann sums of the function $\frac{1}{1-x+x^2}$.
\end{rem}

\subsection{A necessary condition}

All the previous results provide sufficient conditions for the Riemann sums of $f$ to be monotone. Here, we also provide a necessary condition. To prove this result, we need to establish the following lemma.

\begin{lem}\label{lem - constant}
    If $f:[0,1]\to \mathbb{R}$ is a monotonic function and either $R_n(f)$ or $L_n(f)$ is constant, then $f$ is constant, except possibly in $x=0$ or $x=1$.
\end{lem}
\begin{proof}
    Without loss of generality, consider only the case where $R_n(f)$ is constant and $f$ is monotonically increasing. In this case, for all $n$ we have
    \begin{align*}
        f(1) \,=\, R_1(f) \,=\, R_n(f) \,=\, \frac{1}{n} \sum_{k=1}^n f\!\left(\tfrac{k}{n}\right) \,\leq\, \frac{1}{n} \sum_{k=1}^n f(1) \,=\, f(1).
    \end{align*}
    Thus, we have equality in each of the above inequalities, i.e. $f\big(\frac{m}{n}\big)$ for any integers $n\geq 1$ and $1\leq m \leq n$. Since $f$ is monotonically increasing and $f$ is constant on $\mathbb{Q}\cap(0,1]$, it follows that $f$ is constant on $(0,1]$.
\end{proof}

\begin{thm}
    Let $f:[0,1]\to\mathbb{R}$ be a non-constant function. 
    \begin{enumerate}
        \item If $f$ is increasing, then $R_n(f)$ is not increasing and $L_n(f)$ is not decreasing;
        \item If $f$ is decreasing, then $R_n(f)$ is not decreasing and $L_n(f)$ is not increasing.
    \end{enumerate}
\end{thm}
\begin{proof}
    Considering $-f(x)$, $f(1-x)$ and $-f(1-x)$, it suffice to show that if $f$ is increasing, then $R_n(f)$ is not increasing. 
    Hence, let us assume that $R_n(f)$ is increasing. On the one hand, we have
    \begin{equation*}
        R_n(f)-f(1) \,=\, \frac{1}{n} \sum_{k=1}^n f\big(\tfrac{k}{n}\big)-f(1) \,\leq\, \frac{1}{n} \sum_{k=1}^n f(1)-f(1) \,=\, 0,
    \end{equation*}
    and on the other, we have
    \begin{equation*}
        R_n(f)-f(1) \,\geq\, R_1(f)-f(1) \,=\, 0.
    \end{equation*}
    Therefore, $R_n(f)\equiv f(1)$. By \Cref{lem - constant}, this is a contradiction.
\end{proof}

\section{The function \texorpdfstring{$f_b$}{fb}}

\subsection{Preliminary comments and main results}

Let us recall that, combining the results of Szilárd \cite{Szilard}, Borwein \emph{et al.} \cite{Borwein} and ourselves \cite{Bouthat}, it has already been shown that 
\begin{enumerate}
    \item $L_n(f_b)$ is monotonically decreasing for all $b\in(-\infty,0.493862)$;
    \item $R_n(f_b)$ is monotonically increasing for all $b\in\big(-\infty,\tfrac{1}{2}\big)$,
\end{enumerate}
where
\[
f_b(x) \,=\, \frac{1}{1-bx+x^2}.
\]
More precisely, it was shown in \cite{Bouthat} that $R_n(f_b)$ is monotonically increasing for all $b\leq 1$, but the proof contained a mistake. In this section, we give a corrected proof of this assertion, and we further show the following theorem.
\vspace{-2.5pt}

\begin{thm}\label{thm - main_fb}
    Let $f_b(x):=\frac{1}{1-bx+x^2}$ and $\beta^{+}=\frac{3+\sqrt{13}}{4}\approx 1.651388$. Then 
    \begin{enumerate}
        \item $L_n(f_b)$ is monotonically decreasing for all $b\in\big(-\infty, \frac{1}{2}\big]$ and monotonically increasing for all $b\in[1,\beta^+]$; 
        \item $R_n(f_b)$ is monotonically increasing for all $b\in(-\infty, 1]$ and monotonically decreasing for all $b\in\big[\frac{3}{2},\beta^+\big]$.
    \end{enumerate}
    \vspace{-2.5pt}
\end{thm}

Since the study of the monotonicity of Riemann sums of $f_b$ only makes sense for $b\leq 2$, the above result is not far from optimal. However, it says nothing about the monotonicity of Riemann sums when $b\in (\beta^+,2)$, nor when $b\in \big(\frac{1}{2},1\big)$ in the case of $L_n(f_b)$ and when $b\in\big(1,\frac{3}{2}\big)$ in the case of $R_n(f_b)$. While we are not able to provide significant results in these cases at the moment, we do have the following conjecture:
\vspace{-1pt}

\begin{conj}\label{conj}
    Let $f_b(x):=\frac{1}{1-bx+x^2}$. Then 
    \begin{enumerate}
        \item $L_n(f_b)$ is neither increasing nor decreasing for all $b\in\big(\frac{1}{2},1\big)$ and is monotonically increasing for all $b\in(\beta^+,2]$; 
        \item $R_n(f_b)$ is neither increasing nor decreasing for all $b\in\big(1,\frac{3}{2}\big)$ and is monotonically decreasing for all $b\in (\beta^+,2)$.
    \end{enumerate}
    \vspace{-4.5pt}
\end{conj}

\begin{rem}\label{rem - b=2}
    When $b=2$, $L_n(f_b)=n\big(\frac{\pi^{2}}{6}-\psi_1(n+1)\big)$, where $\psi_1(z)$ is the trigamma function, which is a product of increasing function. Hence, \Cref{conj}.1. is verified when $b=2$.
\end{rem}

The proof of the statements in \Cref{thm - main_fb} is naturally divided into two main parts: the cases $b\neq 1$ and $b=1$. Before addressing these, we first need to recall a technical lemma about the convexity of the symmetrization of $f_b$, due to Borwein \emph{et~al.}.

\begin{lem}\textup{\cite{Borwein}}\label{lem - conc}
    Let $f_b(x)=\frac{1}{1-bx+x^2}$ $(b< 2)$, and let $\alpha \approx -0.8794$ be the negative root of $b^3-3b^2+3$, $\gamma:=1+2\sin\left(\frac{\pi}{18}\right) \approx 1. 347296$ and $\beta^{\pm}=\frac{3\pm\sqrt{13}}{4}\approx -0.151388$ and $1.651388$. Then the symmetrization $\mathcal{F}_b(x)$ of $f_b$ is concave for any $b\in (-\infty,\alpha]\cup[\beta^{-},1]\cup[\gamma,\beta^{+}]$ and has a point of inflection for any other $b<2$.
\end{lem}

\subsection{Proof of \texorpdfstring{\Cref{thm - main_fb} when $b\neq 1$}{Theorem 3.1 when b≠1}}

\subsubsection{The case \texorpdfstring{$b\leq \tfrac{1}{2}$}{b≤½}}

It was shown in \cite{Borwein} that $R_n(f_b)$ is increasing for all $b\leq \tfrac{1}{2}$. Moreover, it was also proved in \cite{Bouthat} that $L_n(f_b)$ is decreasing for all $b<\alpha \approx 0.493862$. Using our new result on symmetrization (i.e., \Cref{thm - improvement1}), we can now prove the following theorem, verifying our previously stated conjecture that $L_n(f_b)$ is decreasing for all $b\leq \frac{1}{2}$.

\begin{thm}
    Let $f_b(x)=\frac{1}{1-bx+x^2}$. Then $L_n(f_b)$ is monotonically decreasing for all $b\leq \frac{1}{2}$.
\end{thm}
\begin{proof}
    The result is already verified for $b\leq 0$. Hence, without any loss of generality, assume that $b \geq 0$. By \Cref{lem - conc}, the symmetrization of $f_b$ is concave for all $b\in[0,1]$. Moreover, it is easily verified that $f_b(0)=1 \geq \frac{4}{5-2b} = f_b\big(\frac{1}{2}\big)$ if $b\leq \frac{1}{2}$.
    Consequently, \Cref{thm - improvement1} ensures that $L_n(f_b)$ is decreasing for all $b\leq \frac{1}{2}$.
\end{proof}

\begin{rem}
    If $b=0$, the above proof also shows that $R_n(f_0)=R_n\big(\frac{1}{1+x^2}\big)$ is increasing, giving a simple proof to the initial question of Szilárd \cite{Szilard}.
\end{rem}

\subsubsection{The case \texorpdfstring{$\tfrac{3}{2} \leq b\leq \beta^+$}{3/2≤b≤β+}}

In this case, we have the following result.

\begin{thm}
    Let $f_b(x)=\frac{1}{1-bx+x^2}$. Then $L_n(f_b)$ is monotonically increasing and $R_n(f)$ is monotonically decreasing for all $b\in\big[\frac{3}{2},\beta^+\big]$.
\end{thm}
\begin{proof}
    By \Cref{lem - conc}, the symmetrization of $f_b$ is concave for $b\in\big[\frac{3}{2},\beta^+\big]$. Moreover, it is easily shown that $f\big(\frac{1}{2}\big)=\frac{4}{5-2b} \leq \frac{1}{2-b} = f(1)$ if $\frac{3}{2} \leq b < 2$. Therefore, \Cref{thm - improvement1} ensures that $L_n(f_b)$ increases monotonically and $R_n(f_b)$ decreases monotonically for $b\in\big[\frac{3}{2},\beta^+\big]$.
\end{proof}

Numerically, $R_n(f_b)$ appears to exhibit non-monotonic behavior in the interval $b\in\big(1,\frac{3}{2}\big)$. Hence, the above result appears to be optimal \emph{on the left of the interval} for the right Riemann sum.

\smallskip

\subsubsection{The case \texorpdfstring{$\gamma \leq b\leq \tfrac{3}{2}$}{γ≤b≤3/2}}

In this case, $R_n(f_b)$ appears to be non-monotonic. Hence, let us focus on the left Riemann sum of $f_b$.

\begin{thm}\label{thm - estim_gauche_sup}
    \hspace{-1.4pt}Let $f_b(x)\hspace{-1.3pt}=\hspace{-1.3pt}\frac{1}{1-bx+x^2}$\hspace{-.2pt}. \hspace{-1pt}Then $L_n\hspace{-.5pt}(f_b)$\hspace{-1pt} is monotonically increasing for all $b\!\hspace{.5pt}\in\!\hspace{-.5pt}\big[\hspace{-.5pt}\gamma,\frac{3}{2}\big]$\hspace{-1pt}.
\end{thm}
\vspace{-5pt}
\begin{proof}
    Consider the $4^{\text{th}}$-order Taylor polynomial of $f_1(x)$ at the point $x=\tfrac{1}{2}$, i.e.,
    \[
    h(x) \,:=\, \tfrac{4}{3}\!\left( 1-\left(\tfrac{2x-1}{\sqrt{3}}\right)^{\!2}+\left(\tfrac{2x-1}{\sqrt{3}}\right)^{\!4}\right),
    \]
    and define
    \vspace{-1pt}
    \[
    g(x) \,:=\, f_b(x) - \frac{f_b(1)-f_b\big(\tfrac{1}{2}\big)}{h(1)-h\big(\frac{1}{2}\big)} h(x) .
    \]
    It is straightforward to show that the left Riemann sum of $h$ is monotonically increasing, since $L_n(h)=\frac{4}{27}\left(\frac{41}{5}-\frac{2}{3n^{2}}-\frac{8}{15n^{4}}\right)$. It is also easy to verify that $t: = \frac{f_b(1)-f_b(\frac{1}{2})}{h(1)-h(\frac{1}{2})} = \frac{27}{8}\frac{3-2b}{(5-2b)(2-b)}$ is positive for $b\in \big[\gamma,\frac{3}{2}\big]$. Hence, to prove the desired statement, it suffices to show that $L_n(g)$ is also increasing since $L_n(f_b)= L_n(g)+t L_n(h)$. 

    \smallskip

    To do this, consider the symmetrization of $g$, $\mathcal{G}(x) = \mathcal{F}_{b}(x)-th(x)$, whose second derivative is equal to
    \begin{align*}
        \mathcal{G}''(x) \,=\,  \mathcal{F}''_{b}(x)-th''(x) \,=\, \mathcal{F}''_{b}(x)-\frac{32t\big(8\big(x-\frac{1}{2}\big)^{2}-1\big)}{9}.
    \end{align*}
    We will show that $\mathcal{G}''(x) \leq 0$ and it will thus follow from \Cref{thm - improvement1} that $L_n(g)$ is increasing, since $g(1)=g\big(\frac{1}{2}\big)$ by construction. Hence, let us first observe that if $8\big(x-\frac{1}{2}\big)^2\geq 1$, then by \Cref{lem - conc} we have
    \begin{align*}
        \mathcal{G}''(x) \,=\, \mathcal{F}''_{b}(x)-\frac{32t\big(8\big(x-\frac{1}{2}\big)^{2}-1\big)}{9} \,\leq\, \mathcal{F}''_{b}(x) \,\leq\, 0.
    \end{align*}
    We can therefore suppose without loss of generality that $8\big(x-\frac{1}{2}\big)^2\leq 1$. Furthermore, to simplify the following calculations, define
    \begin{gather*}
        X:= 4\big(x-\tfrac{1}{2}\big)^2, ~~\quad A:=1+6b-4b^{2}-3X, \\
        B:=4X(1-b)^{2} ~~\quad\&~~\quad C:=X+5-2b
    \end{gather*}
    and observe that
    \vspace{-2pt}
    \begin{gather}
        \tfrac{1}{2} \,\leq\, X \,\leq\, 1, \\
        C \,\geq\, X+5-2\big(\tfrac{3}{2}\big) = X+2 \,\geq\, 0. \label{eq - C}
    \end{gather}
    It is not hard to show that we then have 
    \begin{equation*}
        \mathcal{F}''_{b}(x) \,=\, -32\,\frac{3B\big(B+3C^{2}+AC\big)+AC^{3}}{\left(C^{2}-B\right)^{3}}.
    \end{equation*}
    Moreover, by \eqref{eq - C}, we have
    \begin{align}
        C^{2}-B \,&\ge\, (X+2)^{2}-B \,=\, (X+2)^{2}-4X(1-b)^{2} \nonumber\\
        &\ge\, (X+2)^{2}-4X\big(1-\tfrac{3}{2}\big)^{2} \,=\, X^{2}+3X+4 \label{eq - pos}\\
        &\geq\, 0. \nonumber
    \end{align}
    Hence, it follows from \Cref{lem - conc}, since $\mathcal{F}_b''(x)\leq 0$, that both the denominator and the numerator of $\mathcal{F}''_{b}$ are positive. Consequently, since $B\geq0$, we have
    \begin{align*}
        \frac{\mathcal{F}''_{b}(x)}{32} \,&=\, -\frac{3B(B+3C^{2}+AC)+AC^{3}}{(C^{2}-B)^{3}} \leq\, -\frac{3B(B+3C^{2}+AC)+AC^{3}}{(C^{2})^{3}} \\[-1pt]
        &=\, -\frac{3B(B+3C^{2}+AC)+AC^{3}}{C^{6}} \,\leq\, -\frac{3B(3C^{2}+AC)+AC^{3}}{C^{6}} \\[-1pt]
        &=\, -\frac{3B(3C+A)+AC^{2}}{C^{5}} \,=\, -\frac{A+12(4-b^{2})\frac{B}{C^{2}}}{C^{3}}.
    \end{align*}
    Now, since
    $C^{2} = (X+5-2b)^{2} \le (1+5-2b)^{2} = 4(3-b)^{2},$ it follows that
    \begin{align*}
        \frac{\mathcal{F}''_{b}(x)}{32} \,&\leq\, -\frac{A+12(4-b^{2})\frac{B}{C^{2}}}{C^{3}} \,\le\, -\frac{A+\frac{12(4-b^{2})}{4(3-b)^{2}}B}{C^{3}} \\
        &=\, \frac{3\Big(1-4\frac{(4-b^{2})(1-b)^{2}}{(3-b)^{2}}\Big)X+4b^{2}-6b-1}{C^3}.
    \end{align*}
    It is not hard to show that $1\geq 4\frac{(4-b^{2})(1-b)^{2}}{(3-b)^{2}} $ when $b\in \big[\gamma,\frac{3}{2}\big]$. Therefore, since $X\leq 1$,
    \begin{align*}
        3\Big(1-4\tfrac{(4-b^{2})(1-b)^{2}}{(3-b)^{2}}\Big)X+4b^{2}-6b-1 \,&\le\, 3\Big(1-4\tfrac{(4-b^{2})(1-b)^{2}}{(3-b)^{2}}\Big)+4b^{2}-6b-1 \\
        &=\, \frac{2(b-1)(8b^{3}-19b^{2}+15)}{(3-b)^{2}} \,<\, 0.
    \end{align*}
    Thus, since $C=X+5-2b \leq \frac{1}{2}+5-2b = \frac{11}{2}-2b$, it follows that 
    \begin{align*}
        \frac{\mathcal{F}''_{b}(x)}{32} \,&\leq\, \frac{3\Big(1-4\frac{(4-b^{2})(1-b)^{2}}{(3-b)^{2}}\Big)X+4b^{2}-6b-1}{C^3} \\
        &\leq\, \frac{3\Big(1-4\frac{(4-b^{2})(1-b)^{2}}{(3-b)^{2}}\Big)X+4b^{2}-6b-1}{\left(\frac{11}{2}-2b\right)^{\!3}}.
    \end{align*}
    Consequently, $\mathcal{G}''(x) = \mathcal{F}''_{b}(x)-\frac{32t(2X-1)}{9}$ satisfy
    \begin{align*}
        \mathcal{G}''(x) \,&\leq\, 32\,\frac{3\Big(1-4\frac{(4-b^{2})(1-b)^{2}}{(3-b)^{2}}\Big)X+4b^{2}-6b-1}{\left(\frac{11}{2}-2b\right)^{\!3}}-\frac{32t(2X-1)}{9} %\\
        \,=:\, 32\,T(x).
    \end{align*}
    The latter expression is linear in $X$. Hence, the maximum of $T$ is achieved at one of the extremum of the domain of $X$, i.e. at $X=\frac{1}{2}$ or $X=1$. Therefore, it suffices to verify that $\smash{T\big(\tfrac{1}{2}\big)}$ and $T(1)$ are both positive, that is 
    \begin{gather*}
        T\big(\tfrac{1}{2}\big) =\, 8\,\frac{(3-b)^{2}(4b^{2}-6b+\frac{1}{2})-6(4-b^{2})(1-b)^{2}}{(11-4b)^{3} (3-b)^{2}} \,\leq\, 0, \\[-3pt]
        T(1) \,=\, \frac{8(4b^{2}-6b-1)}{(11-4b)^{3}}+\frac{3(3-2b)}{8(5-2b)(2-b)} \,\leq\, 0, 
    \end{gather*}
    for all $b\in \big[\gamma, \frac{3}{2}\big]$. To simplify the equations, put $w:=2(b-1)$. We then find that above conditions are satisfied if and only if, for all $w\in \big[4\sin\!\big(\frac{\pi}{18}\big),1\big]$, we have
    \begin{gather}
        T\big(\tfrac{1}{2}\big) \,=\, 4\, \frac{5w^{4}-2w^{3}-23w^{2}+56w-48}{\left(4-w\right)^{2}\left(7-2w\right)^{3}} \,\leq\, 0, \label{eq - T1/2}\\[-2pt]
        T(1) \,=\, \frac{56w^{4}-404w^{3}+1070w^{2}-1239w+453}{4(3-w)(2-w)(7-2w)^{3}} \,\leq\, 0. \label{eq - T1}
    \end{gather} 
    Let us first verify that \eqref{eq - T1/2} is satisfied. To do this, note that since $w\leq 1$, we have
    \begin{align*}
        T\big(\tfrac{1}{2}\big) \,&=\, 4\, \frac{5w^{4}-2w^{3}-23w^{2}+56w-48}{\left(4-w\right)^{2}\left(7-2w\right)^{3}} \,\leq\, 4\,\frac{5w^{3}-2w^{3}-23w^{2}+56w-48}{\left(4-w\right)^{2}\left(7-2w\right)^{3}} \\
        &=\, 4\,\frac{3w^{3}-23w^{2}+56w-48}{\left(4-w\right)^{2}\left(7-2w\right)^{3}} \,\leq\, 4\,\frac{3w^{2}-23w^{2}+56w-48}{\left(4-w\right)^{2}\left(7-2w\right)^{3}} \\
        &=\, -16\,\frac{5w^{2}-14w+12}{\left(4-w\right)^{2}\left(7-2w\right)^{3}} \,\leq\, -16\,\frac{5w^{2}-14w+12w}{\left(4-w\right)^{2}\left(7-2w\right)^{3}} \\
        &=\, -16\,\frac{w(5w-2)}{\left(4-w\right)^{2}\left(7-2w\right)^{3}} \,\leq\, -16\,\frac{w(5\cdot 4\sin(\frac{\pi}{18})-2)}{\left(4-w\right)^{2}\left(7-2w\right)^{3}} \,\leq\, 0.
    \end{align*}
    Hence, \eqref{eq - T1/2} is satisfied. Let us now show that \eqref{eq - T1} is also true. Since $w\leq 1$, we have
    \begin{align*}
        T(1) \,&=\, \frac{56w^{4}-404w^{3}+1070w^{2}-1239w+453}{4(3-w)(2-w)(7-2w)^{3}} \\
        &\leq\, \frac{56w^{3}-404w^{3}+1070w^{2}-1239w+453}{4(3-w)(2-w)(7-2w)^{3}} \\
        &=\, \frac{-348w^{3}+1070w^{2}-1239w+453}{4(3-w)(2-w)(7-2w)^{3}} \,=:\, \frac{R(w)}{Q(w)}. %\\
        %
        %&=:\, \frac{R(w)}{4(3-w)(2-w)(7-2w)^{3}}
    \end{align*}
    Clearly, $Q(w)> 0$ and an easy computation yield
    \[
    R'(w) \,=\, -1044\left(w-\tfrac{535}{522}\right)^{2}-\tfrac{37154}{261} \,\leq\, 0.
    \]
    Hence, the numerator $R(w)$ is decreasing and thus satisfies
    \begin{align*}
        -348w^{3}&+1070w^{2}-1239w+453 \\
        &\leq\, -348\left(4\sin\!\left(\tfrac{\pi}{18}\right)\right)^{3}+1070\left(4\sin\!\left(\tfrac{\pi}{18}\right)\right)^{2}-1239\left(4\sin\!\left(\tfrac{\pi}{18}\right)\right)+453 \\
        &\approx\, -7.98836 \,<\, 0.
    \end{align*}
    Consequently, \eqref{eq - T1/2} and \eqref{eq - T1} are both satisfied and it follows that the symmetrization of $g$ is concave. By \Cref{thm - improvement1}, we finally deduce that $L_n(g)$ is increasing and thus that $L_n(f_b)$ is monotonically increasing for all $b\in \big[\gamma,\frac{3}{2}\big]$.
\end{proof}

\subsubsection{The case \texorpdfstring{$\frac{1}{2}\leq b \leq \gamma$}{½≤b≤γ}}\label{sec - 1}

To deal with this last case, we need the following technical lemma.

\begin{lem}\label{lem - sym_g}
    Let
    \[
    g(x) \,:=\, f_b(x)-\frac{f_b(0)-f_b\big(\frac{1}{2}\big)}{f_1(0)-f_1\big(\frac{1}{2}\big)} f_1(x).
    \]
    Then the symmetrization of $g$ is concave for all $b\in \big[\frac{1}{2},1\big]$ and convex for all $b\in [1,\gamma]$.
\end{lem}
\vspace{-6pt}
\begin{proof}
    The symmetrization of $g$ is equal to $\mathcal{F}_b(x)-tf_1(x)$, where $t:=\frac{f_b(0)-f_b(\frac{1}{2})}{f_1(0)-f_1(\frac{1}{2})}$. As in the proof of \Cref{thm - estim_gauche_sup}, define 
    \begin{gather*}
        X:= 4\big(x-\tfrac{1}{2}\big)^2, ~~\quad A:=1+6b-4b^{2}-3X, \\[-1pt]
        B:=4X(1-b)^{2} ~~\quad\&~~\quad C:=X+5-2b.
    \end{gather*}
    We then have 
    \begin{align*}
        \mathcal{G}''(x) \,&=\, 32\bigg(\frac{3B\big(B+AC+3C^{2}\big)+AC^{3}}{\left(B-C^{2}\right)^{3}}+\frac{3t\left(1-X\right)}{\left(3+X\right)^{3}}\bigg) \\
        &=\, \frac{128\left(b-1\right)\big(c_{7}X^{7}+c_{6}X^{6}+c_{5}X^{5}+c_{4}X^{4}+c_{3}X^{3}+c_{2}X^{2}+c_{1}X+c_{0}\big)}{\left(5-2b\right)\left(C^{2}-B\right)^{3}\left(3+X\right)^{3}},
    \end{align*}
    where 
    \begin{align*}
        c_{0} \,&:=\, 9\left(5-2b\right)^{4}\big(2b^{2}-6b+7\big), \\
        c_{1} \,&:=\, 18\left(5-2b\right)^{2}\big( -\!16b^{4}+128b^{3}-279b^{2}+157b+37\big), \\
        c_{2} \,&:=\, 9\left(5-2b\right)^{2}\big(48b^{4}-128b^{3}+44b^{2}+20b+43\big), \\
        c_{3} \,&:=\, -1152b^{6}+6784b^{5}-12624b^{4}+3296b^{3}+12740b^{2}-9180b-674, \\
        c_{4} \,&:=\, 288b^{6}-672b^{5}-384b^{4}-672b^{3}+5946b^{2}-3198b-2523, \\
        c_{5} \,&:=\, -216b^{4}+288b^{3}+618b^{2}-246b-930, \\
        c_{6} \,&:=\, 52b^{2}-12b-127, \\
        c_{7} \,&:=\, -6.
    \end{align*}
    First, we will show that for any $b\in \big[\frac{1}{2},\gamma\big]$, $c_0$ and $c_1$ are always positive, while $c_3, c_4, c_5, c_6$ and $c_7$ are always negative. Since $\big[\frac{1}{2},\gamma\big] \subseteq \big[\frac{1}{2},\frac{3}{2}\big]$, in most cases we will consider the latter interval for simplicity.

    \medskip
    \noindent\textbf{Case 1 : $\bm{c_0 \geq 0}$.}
    It suffice to see that we can write
    \[
    c_0 \,=\, \tfrac{9}{2}\left(5-2b\right)^{4} \!\big((3-2b)^{2}+5\big).
    \]

    \smallskip
    
    \noindent\textbf{Case 2 : $\bm{c_1 \geq 0}$.}
    We have
    \begin{align*}
        -16b^{4}+128b^{3}-279b^{2}+157b+37 \,&\geq\, -16b^{2}+128b^{3}-279b^{2}+157b+37 \\
        &=\, (128b^{2}-295b+157)b+37
    \end{align*}
    and $128b^{2}-295b+157$ achieves its minimum of $-\frac{6641}{512}$ in $b=\frac{295}{256} \in \big[\frac{1}{2},\gamma\big]$. Therefore, 
    \begin{align*}
        -16b^{4}+128b^{3}-279b^{2}+157b+37 \,&\geq\,  (128b^{2}-295b+157)b+37 \\
        &\geq\, -\tfrac{6641}{512}b+37 \,\geq\, -\tfrac{6641}{512}\gamma+37 \\
        &\approx\, 0.1972 \,>\, 0.
    \end{align*}
    It follows that, 
    \[
    c_{1} \,=\, 18\left(5-2b\right)^{2}\!\big(\!-\!16b^{4}+128b^{3}-279b^{2}+157b+37\big) \,\geq\, 0.
    \]

    \smallskip

    \noindent\textbf{Case 3 : $\bm{c_3 \leq 0}$.}
    To simplify the proof, put $w:= b-\frac{1}{2} \in [0,1]$ to obtain 
    \[
    c_3 = -2\big(576w^{6}-1664w^{5}-8w^{4}+3936w^{3}-3074w^{2}-812w+1131\big) =: T_3(w).
    \]
    We have
    \begin{align*}
        T_3'(w) \,&=\, -8\big(864w^{5}-2080w^{4}-8w^{3}+2952w^{2}-1537w-203\big) \\
        &\geq\, -8\big(864w^{5}-2080w^{4}+2952w^{2}-1537w-203\big) \\
        &=\, -8\big(864w^{4}-2080w^{3}+2952w-1537\big)w+1624 \\
        &=:\, R_3(w)w+1624
    \end{align*}
    Moreover, $R_3$ is a convex function since $R_3''(w) = 1536w\left(65-54w\right) \geq 0$. Consequently, $R_3'$ is maximal in the interval $[0,1]$ at the point $x=1$ and thus
    \begin{align*}
        R_3'(w) = -192(144w^{3}-260w^{2}+123) \leq -192(144-260+123) = -1344 \,<\, 0.
    \end{align*}
    Therefore, $R_3$ is a decreasing function which implies that
    \begin{align*}
        T_3'(w) \,&\geq\, R_3(w)w+1624 \,\geq\, R_3(1)w+1624 \,=\, -1592w+1624 \\
        &\geq\, -1592+1624 \,=\, 32 \,>\, 0.
    \end{align*}
    It follows that $T_3$ is an increasing function and thus that
    \begin{align*}
        c_3 \,&=\, T_3(w) \,\leq\, T_3(1) \,=\, %-2\left(576-1664-8+3936-3074-812+1131\right) \\
        %
        %&=\, 
        -170 \,<\, 0.
    \end{align*}

    \smallskip

    \noindent\textbf{Case 4 : $\bm{c_4 \leq 0}$.}
    To simplify the proof, put $w:= b-\frac{1}{2} \in [0,1]$ to obtain
    \[
    c_4 \,=\, 288w^{6}+192w^{5}-984w^{4}-2400w^{3}+3792w^{2}+1896w-2760. 
    \]
    Here, we consider $b\in\big[\frac{1}{2},\gamma\big]$ instead of $b\in\big[\frac{1}{2},\frac{3}{2}\big]$ (and consequently, $w\in\big[0,\gamma-\frac{1}{2}\big]$). We then have
    \begin{align*}
        c_4 \,&=\, 288w^{6}+192w^{5}-984w^{4}-2400w^{3}+3792w^{2}+1896w-2760 \\
        &\leq\, 288w^{4}+192w^{4}-984w^{4}-2400w^{3}+3792w^{2}+1896w-2760 \\
        &=\, -24\, (21 w^4 + 100 w^3 - 158 w^2 - 79 w + 115) \,=:\, -24\, T_4(w).
    \end{align*}
    Since $0\leq w \leq \gamma-\frac{1}{2} < 1$, a simple computation yield
    \begin{align*}
        T_4'(w) \,&=\, 84w^{3}+300w^{2}-316w-79 \,\leq\, 84w+300w-316w-79 \\
        &=\, 68w-79 \,\leq\, 68-79 \,=\, -11 \,<\, 0.
    \end{align*}
    Hence, $T_4$ is decreasing in $[0,1]$ and, in particular, for $0\leq w\leq  \gamma-\frac{1}{2}\leq \frac{7}{8} $. It follows that 
    \begin{align*}
        c_4 \,&\leq\, -24\, T_4(w) \,\leq\, -24\, T_4\big(\gamma-\tfrac{1}{2}\big) \,\leq\, -24\, T_4\big(\tfrac{7}{8}\big) \\
        &=\, -24\left( 21\left(\tfrac{7}{8}\right)^{\!4}+100\left(\tfrac{7}{8}\right)^{\!3}-158\left(\tfrac{7}{8}\right)^{\!2}-79\left(\tfrac{7}{8}\right)+115 \right) \\
        &=\, -\tfrac{51711}{512} \,<\, 0.
    \end{align*}

    \smallskip

    \noindent\textbf{Case 5 : $\bm{c_5 \leq 0}$.}
    To simplify the proof, put $w:= b-\frac{1}{2} \in [0,1]$ to obtain
    \begin{align*}
        c_5 \,&=\, -216w^{4}-144w^{3}+726w^{2}+480w-876 \\
        &\leq\, -216w^{4}-144w^{3}+726w^{2}+480-876 \\
        &=\, -216w^{4}-144w^{3}+726w^{2}-396 \,=:\, T_5(w).
    \end{align*}
    A simple computation yield
    \begin{align*}
        T_5'(w) \,=\, 12w\big(121-36w-72w^{2}\big) \,\geq\, 12w\left(121-36-72\right) \,=\, 156w \,\geq\, 0.
    \end{align*}
    Hence, the maximum of $T_5$ is achieved at $w=1$ and we have
    \begin{align*}
        c_5 \,&\leq\, T_5(w) \,\leq\, T_5(1) \,=\,  -216-144+726-396 \,=\, -30 \,<\, 0.
    \end{align*}

    \smallskip

    \noindent\textbf{Case 6 : $\bm{c_6 \leq 0}$.}
    To simplify the proof, put $w:= b-\frac{1}{2} \in [0,1]$ to obtain
    \[
    c_5 \,=\, 52w^{2}+40w-120 \,\leq\, 52+40-120 \,=\, -28 \,<\,0. 
    \]

    \smallskip

    \noindent\textbf{Case 7 : $\bm{c_7 \leq 0}$.}
    Trivial.

    \medskip

    \noindent Therefore, we have $c_0,c_1\geq 0$ and $c_3, c_4, c_5, c_6, c_7\leq 0$ for all $b\in \big[\frac{1}{2},\gamma\big]$ and it follows that
    \begin{align*}
        c_{7}X^{7}&+c_{6}X^{6}+c_{5}X^{5}+c_{4}X^{4}+c_{3}X^{3}+c_{2}X^{2}+c_{1}X+c_{0} \\
        %
        %&\geq\, c_{7}X^{2}+c_{6}X^{2}+c_{5}X^{2}+c_{4}X^{2}+c_{3}X^{2}+c_{2}X^{2}+c_{1}X+c_{0} \\
        %
        &\geq\, c_{7}X^{2}+c_{6}X^{2}+c_{5}X^{2}+c_{4}X^{2}+c_{3}X^{2}+c_{2}X^{2}+c_{1}X^2+c_{0}X^2 \\
        &=\, (c_{7}+c_{6}+c_{5}+c_{4}+c_{3}+c_{2}+c_{1}+c_{0})X^2 \\
        &=\, 2048\left(5-2b\right)\left(2-b\right)\big(b^{3}-3b^{2}+3\big)X^2.
    \end{align*}
    Moreover, recall the inequality \eqref{eq - pos} ensuring the positivity of $C^2-B$. Hence, we find that 
    \begin{align*}
        R(x) \,:=\, \frac{c_{7}X^{7}+c_{6}X^{6}+c_{5}X^{5}+c_{4}X^{4}+c_{3}X^{3}+c_{2}X^{2}+c_{1}X+c_{0}}{\left(5-2b\right)\left(C^{2}-B\right)^{3}\left(3+X\right)^{3}} \,\geq\, 0.
    \end{align*}
    Since $\mathcal{G}''(x) = 128(b-1)R(x)$, we can therefore conclude that $\mathcal{G}''(x)$ is negative if $b\leq 1$ and positive if $b\geq 1$. Consequently, the symmetrization of $g$ is concave for all $b\in \big[\frac{1}{2},1\big]$ and convex for all $b\in [1,\gamma]$. 
\end{proof}

In the interval $\frac{1}{2} \leq b \leq 1$ (resp. $1 \leq b \leq \gamma$), $L_n(f_b)$ (resp. $R_n(f_b)$) numerically appears to be non-monotonic. Hence, let us consider only the cases of the right Riemann sum when $\frac{1}{2} \leq b \leq 1$ and of the left Riemann sum when $1 \leq b \leq \gamma$. It turns out that to show the monotonicity of these Riemann sums, it is sufficient to consider only the case of $b=1$.  

\begin{prop}
    Let $f_b(x)=\frac{1}{1-bx+x^2}$. Suppose that $L_n(f_1)=R_n(f_1)$ is increasing. Then $L_n(f_b)$ is monotonically increasing for all $b \in [1,\gamma]$ and $R_n(f_b)$ is monotonically increasing for all $b \in \big[\frac{1}{2},1\big]$.
\end{prop}
\begin{proof}
    Consider the function 
    \[
    g(x) \,:=\, f_b(x)-\frac{f_b(0)-f_b\big(\frac{1}{2}\big)}{f_1(0)-f_1\big(\frac{1}{2}\big)} f_1(x).
    \]
    It is not hard to verify that $t:= \frac{f_b(0)-f_b(\frac{1}{2})}{f_1(0)-f_1(\frac{1}{2})} = \frac{3(2b-1)}{5-2b} \geq 0$ for all $\frac{1}{2}\leq b \leq \gamma $. Hence, since $L_n(f_b)=L_n(g)+tL_n(f_1)$ and $L_n(f_1)$ is assumed to be increasing, we only need to prove that $L_n(g)$ is also increasing. 
    Now, the symmetrization of $g$ is convex for all $b\in[1,\gamma]$ by \Cref{lem - sym_g}. Moreover, we have $g(0)=g\big(\frac{1}{2}\big)$ by construction. Hence, it follows from \Cref{thm - improvement1} that $L_n(g)$ is increasing for all $b\in[1,\gamma]$, and thus that $L_n(f_b)$ is also monotonically increasing for all $b\in[1,\gamma]$.

    Similarly, the symmetrization of $g$ is concave for all $b \in \big[\frac{1}{2},1\big]$ by \Cref{lem - sym_g} and \Cref{thm - improvement1} then ensures that $R_n(g)$ is increasing. Hence, since $R_n(f_b)=R_n(g)+tR_n(f_1)$, it follows that $R_n(f_b)$ is monotonically increasing for all $b \in \big[\frac{1}{2},1\big]$.
\end{proof}

In light of this result, all that remains to show to complete the proof of \Cref{thm - main_fb} is that $L_n(f_1)=R_n(f_1)$ is a monotonically increasing function. Since the proof of this last result is long and technical, it will be presented in the following section.

\subsection{Proof of \texorpdfstring{\Cref{thm - main_fb} when $b=1$}{Theorem 3.1 when b=1}} \label{sec - f1}

\subsubsection{Preliminary estimates}

Before addressing the statement and the proof of the monotonicity of the Riemann sums of $f_1$, we need to prove some technical inequalities.

\vspace{-5pt}
\begin{ineq}\label{lem - ineq4}
    For all $0\leq\alpha\leq 1$ and all $x\geq 0$, 
    \[
    \alpha\sinh(x)+\cosh(x) \,\le\, e^{x}
    \]
\end{ineq}
\vspace{-9pt}
\begin{proof}
    By definition, we have 
    \begin{equation*}
        \alpha\sinh(x)+\cosh(x) = \alpha\frac{e^{x}-e^{-x}}{2}+\frac{e^{x}+e^{-x}}{2} = \frac{1}{2}(1+\alpha)e^{x}+\frac{1}{2}(1-\alpha)e^{-x}.
    \end{equation*}
    Thus, we have the desired inequality if and only if 
    \begin{equation*}
        \tfrac{1}{2}(1+\alpha)+\tfrac{1}{2}(1-\alpha)e^{-2x} \,\leq\, 1.
    \end{equation*}
    Simplifying this last expression, we find that it is equivalent to having $e^{-2x} \leq 1$ (since $\alpha\geq 1$), which is obviously verified.
\end{proof}

\vspace{-5pt}
\begin{ineq}\label{lem - ineq2}
    For all $x> 0,$ we have
    \[
    \coth(x) \,\leq\, \frac{1}{x} + \frac{x}{3}.
    \]
\end{ineq}
\vspace{-9pt}
\begin{proof}
    Writing $\coth(x) = 1+\frac{2}{e^{2x}-1}$, simplifying and setting $x$ instead of $2x$, we find that the desired inequality is equivalent to having $h(x):={\frac{x^{2}+6x+12}{x^{2}-6x+12} e^{-x} \leq 1}$. Differentiating $h$ yield 
    \[
    h'(x) \,=\, -\frac{e^{-x}x^{4}}{\left(x^{2}-6x+12\right)^{2}} \,\leq\, 0
    \]
    and thus, $h(x) \leq h(0) =1$, which completes the proof.
\end{proof}

\vspace{-5pt}
\begin{ineq}\label{lem - ineq3}
    For all $x\geq 0,$ we have
    \[
    \left(x\coth\!\big(\tfrac{x}{2}\big)-2\right)\operatorname{csch}^{2}\!\big(\tfrac{x}{2}\big) \,\leq\, e^{-\frac{x}{3}}.
    \]
\end{ineq}
\vspace{-9pt}
\begin{proof}
    At $x=0$, a direct computation establishes the inequality. If $x>0$, replace $x$ by $2x$ without any loss of generality and observe that the desired inequality is equivalent to having $ 2\frac{x\coth(x)-1}{\sinh^2(x)}e^{\frac{2x}{3}} \leq 1$. Now, \Cref{lem - ineq2} ensures that
    \begin{align*}
        2\,\frac{x\coth(x)-1}{\sinh^2(x)}e^{\frac{2x}{3}} \,&\leq\, 2\,\frac{x\big(\tfrac{1}{x}+\tfrac{x}{3}\big)-1}{\sinh^2(x)} e^{\frac{2x}{3}} \,=\, \frac{2}{3}\bigg(\frac{xe^{\frac{x}{3}}}{\sinh(x)}\bigg)^{\!2}. \label{eq - simple1}
    \end{align*}
    Since $\frac{xe^{\frac{x}{3}}}{\sinh(x)}\geq 0$, the desired inequality is equivalent to having $\frac{xe^{\frac{x}{3}}}{\sinh\left(x\right)}\le\sqrt{\rule{0pt}{8.2pt}\vphantom{y}\hspace{-.5pt}\smash{\frac{3}{2}}}$. Writing $\sinh(x)$ in its exponential form, we further find that this is equivalent to $\frac{xe^{4x}}{e^{7x}-1}\le\frac{1}{\sqrt{24}}$. But it is easy to show that $\frac{xe^{4x}}{e^{7x}-1}\le\frac{xe^{3x}}{e^{6x}-1}$ and it is well known that $\sinh(x)\geq x$ for every $x\geq 0$. Therefore, we finally have
    \[
    \frac{xe^{4x}}{e^{7x}-1} \,\le\, \frac{xe^{3x}}{e^{6x}-1} \,=\, \frac{1}{6}\frac{3x}{\sinh\left(3x\right)} \,\le\, \frac{1}{6} \,<\, \frac{1}{\sqrt{24}},
    \]
    which completes the proof.
\end{proof}

\vspace{-5pt}
\begin{ineq}\label{lem - ineq1}
    For all $x\geq 0$, we have
    \[
    \left(\sinh\left(x\right)-x\right)\operatorname{csch}^{2}\!\big(\tfrac{x}{2}\big) \,\geq\, 2-2e^{-\frac{x}{3}}.
    \]
\end{ineq}
\vspace{-11pt}
\begin{proof}
    If $x=0$, a direct computation establishes the inequality. If $x>0$, the statement is equivalent to having $h(x):=\frac{x+e^{-x}-1}{\cosh(x)-1}e^{\frac{x}{3}}\leq 1$. Differentiating the latter expression yield
    \[
    h'(x) \,=\, -\frac{e^{-\frac{5x}{3}}\left(e^{x}-1\right)\big(4e^{x}\left(x+1\right)+e^{2x}\left(2x-5\right)+1\big)}{6\left(\cosh\left(x\right)-1\right)^{2}}.
    \]
    Using the Taylor expansion of $e^x$, we readily find that
    \[
    4e^{x}\left(x+1\right)+e^{2x}\left(2x-5\right)+1 = \sum_{n=4}^{\infty}\left(4n+4+n2^{n}-5\cdot2^{n}\right)\frac{x^{n}}{n!} =:\, \sum_{n=4}^{\infty}a_n\frac{x^{n}}{n!}.
    \]
    A direct calculation shows that $a_4 = 4>0$, while for $n\geq 5$
    \[
    a_n=4n+4+n2^{n}-5\cdot2^{n} \,\geq\, 4n+4+5\cdot2^{n}-5\cdot2^{n} = 4n+4 \geq 0.
    \]
    Hence, $a_n\geq 0$ for all $n\geq 4$ and it follows that $4e^{x}\left(x+1\right)+e^{2x}\left(2x-5\right)+1\geq0$. Therefore, $h'(x)\leq 0$ for $x>0$ and thus, $h(x) \leq \lim_{x\to 0^+} h(x) = 1$, which concludes the proof.
\end{proof}

\vspace{-4pt}
\smallskip
\subsubsection{Reformulating \texorpdfstring{$R_n(f_1)$}{Rn(f1)}}

We are now ready to address the proof of the monotonicity of the Riemann sums of $f_1$. For completeness, let us first state formally the result.

\vspace{-2pt}
\begin{thm}\label{thm - fb_1}
    Let $f_b(x)=\frac{1}{1-bx+x^2}$. Then $L_n(f_1)=R_n(f_1)$ is monotonically increasing.
\end{thm}
\vspace{-2pt}

Our first step in the proof is to express $R_n(f_1)$ as a sum of functions which are somehow simpler to study. To do this, first note that we can write 
\begin{equation}
    \smash{R_n(f_1) \,=\, \frac{1}{n}\sum_{k=1}^n f_1\big( \tfrac{k}{n}\big) \,=\, \frac{1}{n}\sum_{k=1}^\infty \left[ f_1\big( \tfrac{k}{n}\big) - f_1\big( \tfrac{k}{n}+1\big) \right].}
\end{equation}
Indeed, both $\sum_{k=1}^\infty f_1\!\left( \frac{k}{n}\right)$ and $\sum_{k=1}^\infty f_1\!\left( \frac{k}{n}+1\right)$ are convergent for all $n$ and
\[
\smash{    \sum_{k=1}^\infty \left[ f_1\big( \tfrac{k}{n}\big) - f_1\big( \tfrac{k}{n}+1\big) \right] \,=\, \sum_{k=1}^\infty  f_1\big( \tfrac{k}{n}\big) -\!\! \sum_{k=n+1}^\infty \!\!f_1\big( \tfrac{k}{n}\big) \,=\, \sum_{k=1}^n f_1\big( \tfrac{k}{n}\big).}
\]
Observe that $\sum_{k=1}^\infty f_1\big( \frac{k}{n}+1\big) = \sum_{k=1}^\infty f_{-1}\big( \frac{k}{n}\big)$ and that we can thus write 
\vspace{-6pt}
\begin{align*}
    R_n(f_1) &= \frac{1}{n}\sum_{k=1}^\infty  f_1\big( \tfrac{k}{n}\big) - \frac{1}{n}\sum_{k=1}^\infty f_{-1}\big( \tfrac{k}{n}\big) \\[-2pt]
    &= \frac{1}{n}\sum_{k=1}^\infty  \left[f_1\big( \tfrac{k}{n}\big) + f_{-1}\big( \tfrac{k}{n}\big) \right] - \frac{2}{n}\sum_{k=1}^\infty f_{-1}\big( \tfrac{k}{n}\big) \\[-2pt]
    &= 2n\sum_{k=1}^\infty \frac{n^2+k^2}{n^4+n^2k^2+k^4} - \frac{2}{n}\sum_{k=1}^\infty f_{-1}\big( \tfrac{k}{n}\big) \\[-2pt]
    %
    %&= n\sum_{k=1}^\infty \frac{n^2+k^2}{n^4+n^2k^2+k^4} + n\!\sum_{k=-\infty}^{-1} \!\frac{n^2+k^2}{n^4+n^2k^2+k^4}  - \frac{2}{n}\sum_{k=1}^\infty f_{-1}\big( \tfrac{k}{n}\big) \\
    %
    &= n\sum_{k=-\infty}^\infty \frac{n^2+k^2}{n^4+n^2k^2+k^4} - \frac{1}{n} - \frac{2}{n}\sum_{k=1}^\infty f_{-1}\big( \tfrac{k}{n}\big).
\end{align*}
The sum $\sum_{k=-\infty}^\infty \frac{n^2+k^2}{n^4+n^2k^2+k^4}$ is a sum over all integers $k$ and the singularities of the function $g_n(z) : = \frac{n^2+z^2}{n^4+n^2z^2+z^4}$ are $\frac{\pm 1 \pm \sqrt{3}i}{2}n$ (in particular, no integer is a singularity). Moreover, $g_n(z)$ is meromorphic and satisfies
\begin{align*}
    |g(z)|^2 \,&=\, \frac{n^{4}+r^{4}+2\sqrt{n^{8}-n^{4}r^{4}+r^{8}}}{3\left(n^{4}-r^{4}\right)^{2}} \,\leq\, \frac{n^{4}+r^{4}+2\sqrt{n^{8}+2n^{4}r^{4}+r^{8}}}{3\left(n^{4}-r^{4}\right)^{2}} \\
    &=\, \frac{n^{4}+r^{4}+2|n^{4}+r^{4}|}{3\left(n^{4}-r^{4}\right)^{2}} \,=\, \frac{n^{4}+r^{4}}{\left(n^{4}-r^{4}\right)^{2}},
\end{align*}
where $r=|z|$. For $r\geq n+1$, we thus have
\begin{align*}
    |g(z)|^2 \,&\leq\, \frac{n^{4}+r^{4}}{\left(n^{4}-r^{4}\right)^{2}} \,\leq\, \frac{(n^{4}+\left(n+1\right)^{4})\left(n+1\right)^{4}}{(\left(n+1\right)^{4}-n^{4})^{2}r^{4}} \,=:\, \frac{M_n^2}{r^4}.
\end{align*}
Consequently, for each $n$ we have that $|g_n(z)| \leq M_n/|z|^2$ for all $|z|\geq n+1$. Thus, by the Residue theorem for sums (see \cite[Theorem 4.4.1]{complex}), we have
\begin{equation*}
    \sum_{k=-\infty}^\infty \frac{n^2+k^2}{n^4+n^2k^2+k^4} \,= \sum_{k=-\infty}^\infty g_n(k) \,=\, - \sum_{\zeta} \mathrm{Res}\left( \pi\cot(\pi z)g_n(z), \zeta\right),
\end{equation*}
where the sum on the right-hand side varies over the singularities of $g_n(z)$. 
Since the singularities are poles of order 1, a direct computation then yield
\begin{equation*}
    - \sum_{\zeta} \mathrm{Res}\left( \pi\cot(\pi z)g_n(z), \zeta\right) \,=\, \frac{2\pi\sinh\!\left(\sqrt{3}\pi n\right)}{\sqrt{3}n\left(\cosh\!\left(\sqrt{3}\pi n\right)-\cos(\pi n)\right)}.
\end{equation*}
Therefore, we find that 
\begin{align}
        R_n(f_1) \,&=\, n\sum_{k=-\infty}^\infty \frac{n^2+k^2}{n^4+n^2k^2+k^4} - \frac{1}{n} - \frac{2}{n}\sum_{k=1}^\infty f_{-1}\big( \tfrac{k}{n}\big) \nonumber \\
        &=\, \frac{2\pi\sinh\!\left(\sqrt{3}\pi n\right)}{\sqrt{3}\left(\cosh\!\left(\sqrt{3}\pi n\right)-\cos(\pi n)\right)} - \bigg(\frac{1}{n} + \frac{2}{n}\sum_{k=1}^\infty f_{-1}\big( \tfrac{k}{n}\big) \bigg) \label{eq - simple} \\
        &=:\, F_1(n)-F_2(n). \nonumber
\end{align}
Note that \eqref{eq - simple} is defined for any $n\in\mathbb{R}$ (and not only for $n\in\mathbb{N}$) and is differentiable with
\begin{equation*}
    \frac{d}{dx} R_x(f_1) \,=\, F_1'(x)-F_2'(x).
\end{equation*}
Therefore, if the derivative of $R_x(f_1)$ is positive for every $x\geq 1$, then $R_x(f_1)$ will be increasing for all $x\geq 1$ and in particular for every integers $x=n$. To show that this is the case, we first estimate $F_1'$ and $F_2'$ separately, and then combine these to bound the derivative of $R_x(f_1)$.

\smallskip
\subsubsection{Estimation of the derivative of \texorpdfstring{$F_1$}{F1}}

Since we are only interested in the monotonicity of $R_x(f_1)$ for $x\geq 1$, we make this assumption in this section and all the following ones. Now, a direct calculation yield 
\begin{align*}%\label{eq - F1'}
    F_1'(x) \,&=\, 2\pi^{2}\frac{1-\frac{1}{\sqrt{3}}\sin\!\left(\pi x\right)\sinh\!\left(\sqrt{3}\pi x\right)-\cos\!\left(\pi x\right)\cosh\!\left(\sqrt{3}\pi x\right)}{\left(\cosh\!\left(\sqrt{3}\pi x\right)-\cos\left(\pi x\right)\right)^{2}} \\
    &\geq\, 2\pi^{2}\frac{1-\frac{1}{\sqrt{3}}\sinh\!\left(\sqrt{3}\pi x\right)-\cosh\!\left(\sqrt{3}\pi x\right)}{\left(\cosh\!\left(\sqrt{3}\pi x\right)-1\right)^{2}}.
\end{align*}
Moreover, an application of \Cref{lem - ineq4} reveal that 
\begin{align*}
    F_1'(x) \,&\geq\, 2\pi^{2}\frac{1-\frac{1}{\sqrt{3}}\sinh\!\left(\sqrt{3}\pi x\right)-\cosh\!\left(\sqrt{3}\pi x\right)}{\left(\cosh\!\left(\sqrt{3}\pi x\right)-1\right)^{2}} \,\geq\,  \frac{2\pi^{2}(1-e^{\sqrt{3}\pi x})}{\left(\cosh\!\left(\sqrt{3}\pi x\right)-1\right)^{2}}.
\end{align*}
Therefore, if $x\geq 1$, we have
\begin{align*}
    F_1'(x) \,&\geq\,  \frac{2\pi^{2}(1-e^{\sqrt{3}\pi x})}{\left(\cosh\left(\sqrt{3}\pi x\right)-1\right)^{2}} \,\geq\, \frac{2\pi^{2}(1-e^{\sqrt{3}\pi x})}{\big(\frac{e^{\sqrt{3}\pi x}}{2}-1\big)^{2}} \,=\, \frac{-8\pi^{2}}{e^{\sqrt{3}\pi x}-2}\bigg(\frac{1}{e^{\sqrt{3}\pi x}-2}+1\bigg) \\
    &\geq\, \frac{-8\pi^{2}}{e^{\sqrt{3}\pi x}-2}\bigg(\frac{1}{e^{\sqrt{3}\pi }-2}+1\bigg) \,=\, \frac{-8\pi^{2}}{1-2e^{-\sqrt{3}\pi x}}\bigg(\frac{1}{e^{\sqrt{3}\pi }-2}+1\bigg)\cdot \frac{1}{e^{\sqrt{3}\pi x}} \\
    &\geq\, -\frac{8\pi^{2}}{1-2e^{-\sqrt{3}\pi}}\bigg(\frac{e^{\sqrt{3}\pi}}{e^{\sqrt{3}\pi}-2}+1\bigg)\cdot\frac{1}{e^{\sqrt{3}\pi x}} \,\geq\, -\frac{80}{e^{\sqrt{3}\pi x}}.
\end{align*}
The second inequality derives from the fact that $\cosh(x) = \frac{e^x+e^{-x}}{2} \geq \frac{e^x}{2}$, the third follows from the fact that $e^{\sqrt{3}\pi x}$ is increasing and the fourth from the decreasing nature of the function $\frac{1}{1-2e^{-x}}$ for $x\geq 1$. The last inequality is only present to simplify the writing. However, it is not a bad one, since $\frac{8\pi^{2}}{1-2e^{-\sqrt{3}\pi}}\Big(\frac{e^{\sqrt{3}\pi}}{e^{\sqrt{3}\pi}-2}+1\Big) \approx 79.995$.

\smallskip
\subsubsection{Estimation of the derivative of \texorpdfstring{$F_2$}{F2}}

Let us now focus on the derivative of $F_2$. Our first step is to notice, using the inverse Laplace transform or via direct computation, that 
$f_{-1}(x)=\frac{1}{1+x+x^2}$ is the Laplace transform of the function $\smash{\frac{2}{\sqrt{3}}} e^{-t/2} \sin\!\smash{\big( \frac{\sqrt{3}t}{2}\big)}$. In other words, we have
\begin{equation}
    f_{-1}(x) \,=\, \frac{2}{\sqrt{3}} \int_0^\infty e^{-t/2} \sin\!\left( \tfrac{\sqrt{3}t}{2}\right) e^{-xt} dt.
\end{equation}
It thus follows from Fubini's theorem that 
\begin{align*}
    F_2(x) \,&=\, \frac{1}{x} + \frac{2}{x}\sum_{k=1}^\infty f_{-1}\big( \tfrac{k}{x}\big) \,=\, \frac{1}{x} + \frac{2}{x}\sum_{k=1}^\infty \frac{2}{\sqrt{3}} \int_0^\infty e^{-t/2} \sin\!\left( \tfrac{\sqrt{3}t}{2}\right) e^{-\frac{tk}{x}} dt \\
    &=\, \frac{1}{x} +  \frac{4}{\sqrt{3}x} \int_0^\infty e^{-t/2} \sin\!\left( \tfrac{\sqrt{3}t}{2}\right) \!\left(\sum_{k=1}^\infty e^{-\frac{tk}{x}} \right) dt \\[-3pt]
    &=\, \frac{1}{x} +  \frac{4}{\sqrt{3}x} \int_0^\infty  \frac{e^{-t/2} \sin\big( \frac{\sqrt{3}t}{2}\big)}{e^{\frac{t}{x}}-1} dt \\
    &=\, \frac{1}{x} +  \frac{8\pi}{3x}\int_{0}^{\infty}\frac{e^{-\frac{\pi u}{\sqrt{3}}}\sin(\pi u)}{e^{\frac{2\pi u}{\sqrt{3}x}}-1}du.
\end{align*}
Moreover, it is easily verified that $\frac{4\pi}{3}\int_{0}^{\infty} e^{-\frac{\pi u}{\sqrt{3}}}\sin(\pi u)du = 1.$ Hence, it follows that
\begin{align*}
    F_2(x) \,&=\, \frac{1}{x} +  \frac{8\pi}{3x}\int_{0}^{\infty}\frac{e^{-\frac{\pi u}{\sqrt{3}}}\sin(\pi u)}{e^{\frac{2\pi u}{\sqrt{3}x}}-1}du \\
    &=\, \frac{8\pi}{3x}\int_{0}^{\infty} \frac{e^{-\frac{\pi u}{\sqrt{3}}}\sin(\pi u)}{2} du + \frac{8\pi}{3x}\int_{0}^{\infty}\frac{e^{-\frac{\pi u}{\sqrt{3}}}\sin(\pi u)}{e^{\frac{2\pi u}{\sqrt{3}x}}-1}du \\
    &=\, \frac{8\pi}{3x}\int_{0}^{\infty} \left( \frac{1}{2} +\frac{1}{e^{\frac{2\pi u}{\sqrt{3}x}}-1} \right) e^{-\frac{\pi u}{\sqrt{3}}}\sin(\pi u) du \\
    &=\, \frac{4\pi}{3x}\int_{0}^{\infty}  \frac{e^{\frac{2\pi u}{\sqrt{3}x}}+1}{e^{\frac{2\pi u}{\sqrt{3}x}}-1}  e^{-\frac{\pi u}{\sqrt{3}}}\sin(\pi u) du \\
    &=\, \frac{4\pi}{3x}\int_{0}^{\infty} \coth\!\left( \tfrac{\pi u}{\sqrt{3} x}\right) e^{-\frac{\pi u}{\sqrt{3}}}\sin(\pi u) du.
\end{align*}
The variation in the sign of the sine function in the latter expression is not convenient. Hence, let us remove this variation by writing
\begin{align*}
    F_2(x) \,&=\, \frac{4\pi}{3x}\int_{0}^{\infty} \coth\!\left( \tfrac{\pi u}{\sqrt{3} x}\right) e^{-\frac{\pi u}{\sqrt{3}}}\sin(\pi u) du \\
    &=\, \frac{4\pi}{3x}\sum_{k=0}^\infty \int_{0}^{1} \coth\!\left( \tfrac{\pi (u+k)}{\sqrt{3} x}\right) e^{-\frac{\pi (u+k)}{\sqrt{3}}}\sin(\pi (u+k)) du \\
    &=\, \frac{4\pi}{3x}\sum_{k=0}^\infty (-1)^k \!\int_{0}^{1} \coth\!\left( \tfrac{\pi (u+k)}{\sqrt{3} x}\right) e^{-\frac{\pi (u+k)}{\sqrt{3}}}\sin(\pi u) du.
\end{align*}
Now, to get rid of the oscillation from the $\pm$ sign, we write
\begin{equation*}
    F_2(x) = \frac{4\pi}{3}\sum_{k=0}^{\infty}\int_{0}^{1}\left(\frac{\coth\!\left(\frac{\pi\left(u+2k\right)}{\sqrt{3}x}\right)}{xe^{\frac{\pi\left(u+2k\right)}{\sqrt{3}}}}-\frac{\coth\!\left(\frac{\pi\left(u+2k+1\right)}{x\sqrt{3}}\right)}{xe^{\frac{\pi\left(u+2k+1\right)}{\sqrt{3}}}}\right)\sin(\pi u) du.
\end{equation*}
Finally, differentiating $F_2$ using this form yield
\begin{equation*}
    F_2'(x) = \frac{2\pi}{3x^2} \sum_{k=0}^{\infty}\int_{0}^{1}\left( g\!\left( \tfrac{2\pi\left(u+2k+1\right)}{\sqrt{3}x} \right)e^{-\frac{\pi}{\sqrt{3}}} - g\!\left( \tfrac{2\pi\left(u+2k\right)}{\sqrt{3}x}\right) \right)e^{-\frac{\pi(u+2k)}{\sqrt{3}}}\sin(\pi u)du,
\end{equation*}
where $g(z)=\left(\sinh(z)-z\right)\operatorname{csch}^{2}\!\left(\frac{z}{2}\right)$. To simplify the writing, define $\smash{\zeta := \frac{2\pi(u+2k)}{\sqrt{3}x}}$ and $\Delta:= \smash{\frac{2\pi}{\sqrt{3}x}}$, and note that we have
\begin{equation*}
    F_2'(x) = \frac{2\pi}{3x^2} \sum_{k=0}^{\infty}\int_{0}^{1} \!\!\left( \frac{g\!\left( \zeta +\Delta \right)-g\!\left( \zeta \right)}{\Delta} \Delta e^{-\frac{\pi}{\sqrt{3}}} - \big(1-e^{-\frac{\pi}{\sqrt{3}}}\big)g( \zeta) \!\right)e^{-\frac{\zeta x}{2}}\sin(\pi u)du.
\end{equation*}
By the mean value theorem, there exists a $c\in[\zeta,\zeta+\Delta]$ such that 
\begin{equation*}
    F_2'(x) = \frac{2\pi}{3x^2} \sum_{k=0}^{\infty}\int_{0}^{1}\left( g'(c) \Delta e^{-\frac{\pi}{\sqrt{3}}} - \big(1-e^{-\frac{\pi}{\sqrt{3}}}\big)g( \zeta) \right)e^{-\frac{\zeta x}{2}}\sin(\pi u)du.
\end{equation*}
Using \Cref{lem - ineq3}, we find that $g'(z) = \left(z\coth\!\left(\frac{z}{2}\right)-2\right) \operatorname{csch}^2\!\left(\frac{z}{2}\right) \leq e^{-\frac{z}{3}}$ and thus we have
\begin{equation*}
    F_2'(x) \leq \frac{2\pi}{3x^2} \sum_{k=0}^{\infty}\int_{0}^{1}\left( \Delta e^{-\frac{\pi}{\sqrt{3}}-\frac{c}{3}} - \big(1-e^{-\frac{\pi}{\sqrt{3}}}\big)g( \zeta) \right)e^{-\frac{\zeta x}{2}}\sin(\pi u)du.
\end{equation*}
Now, since $\Delta e^{-\frac{\pi}{\sqrt{3}}} =\frac{2\pi}{\sqrt{3}x}e^{-\frac{\pi}{\sqrt{3}}}>0$, $e^{-\frac{x}{3}}$ is decreasing in $x$ and $c\in[\zeta,\zeta+\Delta]$, it follows that 
\begin{equation*}
    F_2'(x) \leq \frac{2\pi}{3x^2} \sum_{k=0}^{\infty}\int_{0}^{1}\left( \Delta e^{-\frac{\pi}{\sqrt{3}}-\frac{\zeta}{3}} - \big(1-e^{-\frac{\pi}{\sqrt{3}}}\big)g( \zeta) \right)e^{-\frac{\zeta x}{2}}\sin(\pi u)du.
\end{equation*}
Furthermore, by \Cref{lem - ineq1} we also have
\begin{equation*}
    F_2'(x) \leq \frac{2\pi}{3x^2} \sum_{k=0}^{\infty}\int_{0}^{1}\left( \Delta e^{-\frac{\pi}{\sqrt{3}}-\frac{\zeta}{3}}  - \big(1-e^{-\frac{\pi}{\sqrt{3}}}\big)\big(2-2e^{-\frac{\zeta}{3}}\big) \right)e^{-\frac{\zeta x}{2}}\sin(\pi u)du.
\end{equation*}
Therefore, if we also define $\gamma := 1-e^{-\frac{\pi}{\sqrt{3}}}$ and $\omega:=\frac{\pi}{\sqrt{3}}\left(1+\frac{2}{3x}\right)$, we find by a second application of Fubini's Theorem that
\begin{align*}
    F_2'(x) &\leq \frac{2\pi}{3x^2} \sum_{k=0}^{\infty}\int_{0}^{1}\left( \Delta e^{-\frac{\pi}{\sqrt{3}}-\frac{\zeta}{3}}  - \big(1-e^{-\frac{\pi}{\sqrt{3}}}\big)\big(2-2e^{-\frac{\zeta}{3}}\big) \right)e^{-\frac{\zeta x}{2}}\sin(\pi u)du \\
    &= \frac{4\pi}{3x^{2}} \!\int_{0}^{1} \!\left(\!\left(\gamma+\tfrac{1}{2} \Delta e^{\frac{-\pi}{\sqrt{3}}}\right)\sum_{k=0}^{\infty}e^{-\frac{\zeta}{3}-\frac{\zeta x}{2}}-\gamma\sum_{k=0}^{\infty}e^{-\frac{\zeta x}{2}}\right)\sin(\pi u)du \\
    &= \frac{4\pi}{3x^{2}} \!\int_{0}^{1} \!\left(\!\left(\gamma+\tfrac{1}{2} \Delta e^{\frac{-\pi}{\sqrt{3}}}\right)e^{-\omega u}\sum_{k=0}^{\infty}e^{-2\omega k}-\gamma e^{-\frac{\pi u}{\sqrt{3}}}\sum_{k=0}^{\infty}e^{-\frac{2\pi k}{\sqrt{3}}}\right)\sin(\pi u)du \\
    &= \frac{4\pi}{3x^{2}} \!\int_{0}^{1} \!\Bigg(\frac{\gamma+\tfrac{1}{2} \Delta e^{\frac{-\pi}{\sqrt{3}}}}{1-e^{-2\omega}}e^{-\omega u}-\frac{\gamma e^{-\frac{\pi u}{\sqrt{3}}}}{1-e^{-\frac{2\pi}{\sqrt{3}}}} \Bigg)\sin(\pi u)du \\
    &= \frac{4\pi}{3x^{2}}  \Bigg(\frac{\gamma+\tfrac{1}{2} \Delta e^{\frac{-\pi}{\sqrt{3}}}}{1-e^{-2\omega}} \int_{0}^{1} e^{-\omega u}\sin(\pi u)du-\frac{\gamma}{1-e^{-\frac{2\pi}{\sqrt{3}}}} \int_{0}^{1}e^{-\frac{\pi u}{\sqrt{3}}}\sin(\pi u)du \Bigg).
\end{align*}
It is a matter of simple computations to verify that $\int_{0}^{1}e^{-au}\sin(\pi u)du=\frac{\pi\left(1+e^{-a}\right)}{a^{2}+\pi^{2}}$ for any $a\in\mathbb{R}$. Hence, we finally obtain
\begin{align*}
    F_2'(x) \,&\leq\, \frac{4\pi}{3x^{2}}  \!\left(\frac{\gamma+\tfrac{1}{2} \Delta e^{\frac{-\pi}{\sqrt{3}}}}{1-e^{-2\omega}} \int_{0}^{1} e^{-\omega u}\sin(\pi u)du-\frac{\gamma}{1-e^{-\frac{2\pi}{\sqrt{3}}}} \int_{0}^{1}e^{-\frac{\pi u}{\sqrt{3}}}\sin(\pi u)du \right) \\
    &=\, \frac{4\pi}{3x^{2}}  \!\left(\frac{\gamma+\tfrac{1}{2} \Delta e^{\frac{-\pi}{\sqrt{3}}}}{1-e^{-2\omega}} \frac{\pi\left(1+e^{-\omega}\right)}{\omega^{2}+\pi^{2}}-\frac{\gamma}{1-e^{-\frac{2\pi}{\sqrt{3}}}} \frac{\pi\big(1+e^{\frac{-\pi }{\sqrt{3}}}\big)}{\big(\frac{\pi}{\sqrt{3}}\big)^{\!2}+\pi^{2}} \right) \\
    &=\, \frac{1}{x^{2}}\Bigg(\frac{e^{\frac{\pi}{\sqrt{3}}}+\frac{\pi }{\sqrt{3}x}-1}{e^{\frac{\pi}{\sqrt{3}}}-e^{\frac{-2\pi}{3\sqrt{3}x}}}\frac{4}{\left(1+\frac{2}{3x}\right)^{2}+3}-1\Bigg).
\end{align*}

\smallskip
\subsubsection{Estimation of the derivative of \texorpdfstring{$R_x(f_1)$}{Rx(f1)}}

To show that $R_n(f_1)$ is increasing in $n$, first note that $R_1(f_1) = 1 \leq \frac{7}{6} = R_2(f_1) $. Hence, we can suppose that $n\geq 2$ and thus, in the following, we will assume that $x\geq 2$ and show that $\frac{d}{dx} R_x(f_1) \geq 0$. In the previous sections, we showed that 
\begin{gather}
    \frac{d}{dx} R_x(f_1) \,=\, F_1'(x)-F_2'(x) , \\
    F_1'(x) \,\geq\, -80e^{-\sqrt{3}\pi x}, \\
    F_2'(x) \,\leq\, \frac{1}{x^{2}}\Bigg(\frac{e^{\frac{\pi}{\sqrt{3}}}+\frac{\pi }{\sqrt{3}x}-1}{e^{\frac{\pi}{\sqrt{3}}}-e^{\frac{-2\pi}{3\sqrt{3}x}}}\frac{4}{\left(1+\frac{2}{3x}\right)^{2}+3}-1\Bigg).
\end{gather}
Therefore,
\begin{equation*}
    \frac{d}{dx} R_x(f_1) \,\geq\, -\frac{80}{e^{\sqrt{3}\pi x}} - \frac{1}{x^{2}}\Bigg(\frac{e^{\frac{\pi}{\sqrt{3}}}+\frac{\pi }{\sqrt{3}x}-1}{e^{\frac{\pi}{\sqrt{3}}}-e^{\frac{-2\pi}{3\sqrt{3}x}}}\frac{4}{\left(1+\frac{2}{3x}\right)^{2}+3}-1\Bigg).
\end{equation*}
In order to further simplify the expression on the right-hand side until we are able to show that it is positive, we will use the following numerical inequalities:
\begin{gather*}
    \mathbf{(1)}~\frac{2\sqrt{3}e^{\frac{\pi}{\sqrt{3}}}}{\sqrt{3}+\pi}>4,\quad 
    \mathbf{(2)}~\frac{9(e^{\frac{\pi}{\sqrt{3}}}-1)}{\pi(\sqrt{3}+\pi)}>3,\quad
    \mathbf{(3)}~\frac{27\cdot3^{5}\cdot\pi^{10}}{800\cdot10!}>\frac{1}{5}.
\end{gather*}
Moreover, we will also use the following well-known inequalities (valid for any $x\geq 0$)
\begin{gather*}
    \mathbf{(4)}~e^x\geq \frac{x^n}{n!},\qquad 
    \mathbf{(5)}~e^{-x}<\frac{1}{x+1}.
\end{gather*}
Thus, using $\mathbf{(5)}$ we find
\begin{align*}
    \frac{d}{dx} R_x(f_1) &\geq -\frac{80}{e^{\sqrt{3}\pi x}} - \frac{1}{x^{2}}\Bigg(\frac{e^{\frac{\pi}{\sqrt{3}}}+\frac{\pi }{\sqrt{3}x}-1}{e^{\frac{\pi}{\sqrt{3}}}-e^{\frac{-2\pi}{3\sqrt{3}x}}}\frac{4}{\left(1+\frac{2}{3x}\right)^{2}+3}-1\Bigg) \\
    &\geq -\frac{80}{e^{\sqrt{3}\pi x}} -\frac{1}{x^{2}}\!\left(\frac{e^{\frac{\pi}{\sqrt{3}}}+\frac{\pi }{\sqrt{3}x}-1}{e^{\frac{\pi}{\sqrt{3}}}-\frac{1}{\frac{2\pi}{3\sqrt{3}x}+1}}\frac{4}{\left(1+\frac{2}{3x}\right)^{2}+3}-1\right) \\
    &= -\frac{80}{e^{\sqrt{3}\pi x}}-\frac{1}{x^{2}}\!\left(\!\Bigg(1+\frac{1+\frac{2\pi}{\sqrt{3}x}}{2e^{\frac{\pi}{\sqrt{3}}}+\frac{3\sqrt{3}}{\pi}\big(e^{\frac{\pi}{\sqrt{3}}}-1\big)x}\Bigg)\frac{4}{\left(1+\frac{2}{3x}\right)^{2}+3}-1\right)\!.
\end{align*}
Moreover, since we assumed that $x\geq 2$, we also have 
\begin{align*}
    \frac{d}{dx} R_x(f_1) \,&\geq\, -\frac{80}{e^{\sqrt{3}\pi x}}-\frac{1}{x^{2}}\!\left(\!\Bigg(1+\frac{1+\frac{2\pi}{\sqrt{3}x}}{2e^{\frac{\pi}{\sqrt{3}}}+\frac{3\sqrt{3}}{\pi}\big(e^{\frac{\pi}{\sqrt{3}}}-1\big)x}\Bigg)\frac{4}{\left(1+\frac{2}{3x}\right)^{2}+3}-1\right) \\
    &\geq\, -\frac{80}{e^{\sqrt{3}\pi x}}-\frac{1}{x^{2}}\!\left(\!\Bigg(1+\frac{1+\frac{\pi}{\sqrt{3}}}{2e^{\frac{\pi}{\sqrt{3}}}+\frac{3\sqrt{3}}{\pi}\big(e^{\frac{\pi}{\sqrt{3}}}-1\big)x}\Bigg)\frac{4}{\left(1+\frac{2}{3x}\right)^{2}+3}-1\right) \\
    &=\, -\frac{80}{e^{\sqrt{3}\pi x}}-\frac{1}{x^{2}}\!\left(\!\left(1+\frac{1}{\frac{2\sqrt{3}e^{\frac{\pi}{\sqrt{3}}}}{\sqrt{3}+\pi}+\frac{9(e^{\frac{\pi}{\sqrt{3}}}-1)}{\pi(\sqrt{3}+\pi)}x}\right)\frac{4}{\left(1+\frac{2}{3x}\right)^{2}+3}-1\right)\!.
\end{align*}
Furthermore, it follows from the inequalities $\mathbf{(1)}$ and $\mathbf{(2)}$ that 
\begin{align*}
    \frac{d}{dx} R_x(f_1) \,&\geq\,  -\frac{80}{e^{\sqrt{3}\pi x}}-\frac{1}{x^{2}}\!\left(\!\left(1+\frac{1}{\frac{2\sqrt{3}e^{\frac{\pi}{\sqrt{3}}}}{\sqrt{3}+\pi}+\frac{9(e^{\frac{\pi}{\sqrt{3}}}-1)}{\pi(\sqrt{3}+\pi)}x}\right)\frac{4}{\left(1+\frac{2}{3x}\right)^{2}+3}-1\right) \\
    &\geq\, -\frac{80}{e^{\sqrt{3}\pi x}}-\frac{1}{x^{2}}\!\left(\left(1+\frac{1}{4+3x}\right)\frac{4}{\left(1+\frac{2}{3x}\right)^{2}+3}-1\right) \\
    &=\, -\frac{80}{e^{\sqrt{3}\pi x}}+\frac{15x+4}{x^{2}\left(3x+4\right)\left(9x^{2}+3x+1\right)}.
\end{align*}
Using once again the fact that $x\geq 2$ and the fact that $\frac{15x+4}{3x+4}$ is an increasing function, we find that 
\begin{align*}
    \frac{d}{dx} R_x(f_1) \,&\geq\, -\frac{80}{e^{\sqrt{3}\pi x}}+\frac{15x+4}{x^{2}\left(3x+4\right)\left(9x^{2}+3x+1\right)} \\
    &\geq\, -\frac{80}{e^{\sqrt{3}\pi x}}+\frac{15\cdot2+4}{x^{2}\left(3\cdot2+4\right)\left(9x^{2}+3x+1\right)} \\
    &=\, \frac{80}{e^{\sqrt{3}\pi x}}\left(\frac{17e^{\sqrt{3}\pi x}}{400x^{2}\left(9x^{2}+3x+1\right)}-1\right).
\end{align*}
Let us now use the inequalities $\mathbf{(3)}$ and $\mathbf{(4)}$, with $n=8$, to obtain 
\begin{align*}
    \frac{d}{dx} R_x(f_1) \,&\geq\, \frac{80}{e^{\sqrt{3}\pi x}}\left(\frac{17e^{\sqrt{3}\pi x}}{400x^{2}\left(9x^{2}+3x+1\right)}-1\right) \\
    &\geq\, \frac{80}{e^{\sqrt{3}\pi x}}\left(\frac{17\left(\sqrt{3}\pi x\right)^{8}}{400\cdot 8!\cdot x^{2}\left(9x^{2}+3x+1\right)}-1\right) \\
    &=\, \frac{80}{e^{\sqrt{3}\pi x}}\left(\frac{17\cdot3^{4}\cdot\pi^{8}x^{6}}{400\cdot8!\cdot\left(9x^{2}+3x+1\right)}-1\right) \\
    &\geq\, \frac{80}{e^{\sqrt{3}\pi x}}\left(\frac{4x^{6}}{5\left(9x^{2}+3x+1\right)}-1\right).
\end{align*}
Therefore, to show that $\frac{d}{dx} R_x(f_1)\geq 0$ for each $x\geq 2$, it suffices to prove that $\frac{4x^{6}}{5\left(9x^{2}+3x+1\right)} \geq 1$ for each $x\geq 2$. To do this, simply observe that the derivative of the rational term on the right-hand side of the above inequality is equal to
\[
\frac{12x^{2}\left(12x^{2}+5x+2\right)}{5\left(9x^{2}+3x+1\right)^{2}} \,>\,0.
\]
Hence, the rational term is increasing and we thus have
\begin{align*}
    \frac{4x^{6}}{5\left(9x^{2}+3x+1\right)} \,\geq\, \frac{4\cdot2^{6}}{5\left(9\cdot2^{2}+3\cdot2+1\right)} \,=\, \frac{256}{215} \,>\, 1.
\end{align*}
Therefore, it finally follows that
\begin{align*}
    \frac{d}{dx} R_x(f_1) \,\geq\, 
 \frac{80}{e^{\sqrt{3}\pi x}}\left(\frac{4x^{6}}{5\left(9x^{2}+3x+1\right)}-1\right) \,\geq\, 0
\end{align*}
for each $x\geq 2$. In particular, we finally find that $R_n(f_1)$ is a monotonically increasing function in $n$, for all $n\geq 1$, which concludes the proof of \Cref{thm - fb_1}.

\section{Concluding remarks}

\begin{enumerate}
    \item Statement 3 of \Cref{thm - weird} is more general than Statements 1,2 and 4. This suggests that the latter could possibly be generalized in a similar way to the former. Does such a generalization exist, and if not, why is the case of Statement 3 special?
    \item Can the general theorems about the monotonicity of the left and right Riemann sums in this paper be generalized in the same way as \Cref{thm - weird}? If so, which one?
    \item Can we prove or disprove \Cref{conj}?
\end{enumerate}

\bibliographystyle{plain}
\bibliography{ref}

\end{sloppypar}
\end{document}